\documentclass[11pt]{amsart}
\usepackage{texmac}
\usepackage[all]{xy}
\numberwithin{equation}{section}

\operators{Stab,Sel,Gal,Spec, PGL}
\calsymbols{c}{A,C,D,K,X,Y,F,U,V}
\bbsymbols{}{P,A}

\newcommand{\xb}{\bar{x}}
\newcommand{\yb}{\bar{y}}
\newcommand{\zb}{\bar{z}}
\newcommand{\wb}{\bar{w}}
\newcommand{\zetab}{\bar{\zeta}}
\newcommand{\fb}{\bar{f}}

\newcommand{\hb}{\bar{h}}

\newcommand{\ub}{\bar{u}}

\newcommand{\Xb}{\bar{X}}
\newcommand{\Yb}{\bar{Y}}
\newcommand{\Zb}{\bar{Z}}
\newcommand{\Wb}{\bar{W}}
\newcommand{\Db}{\bar{D}}
\newcommand{\Jb}{\bar{\J}}
\newcommand{\Gammab}{\bar{\Gamma}}

\newcommand{\p}{\mathfrak{p}}

\newcommand{\J}{\mathcal{J}}
\newcommand{\ZZ}{\mathbb{Z}}
\newcommand{\FF}{\mathbb{F}}
\newcommand{\QQ}{\mathbb{Q}}
\newcommand{\OO}{\mathcal O}

\newcommand{\Kb}{{K^\alg}}
\newcommand{\et}{{\rm et}}
\newcommand{\alg}{{\rm alg}}
\newcommand{\nr}{{\rm ur}}
\newcommand{\reg}{{\rm reg}}

\newcommand{\alphab}{\bar{\alpha}}

\newcommand{\lambdab}{\bar{\lambda}}
\newcommand{\phib}{\bar{\phi}}

\newcommand{\liso}{\stackrel{\sim}{\to}}
\newcommand{\lcm}{\mathop{\rm lcm}}

\begin{document}

\title[$L$-functions and semistable reduction]{Computing $L$-functions and semistable reduction of superelliptic
  curves}

\author{Irene I.\ Bouw, Stefan Wewers}

\begin{abstract}
  We give an explicit description of the stable reduction of superelliptic
  curves of the form $y^n=f(x)$ at primes $\p$ whose residue characteristic is
  prime to the exponent $n$. We then use this description to compute the local
  $L$-factor  and the exponent of conductor at $\p$ of the curve. 

\end{abstract}

\maketitle

\section{Introduction}

\subsection{}

Let $Y$ be a smooth projective curve of genus $g\geq 2$ over a number field
$K$. The {\em $L$-function} of $Y$ is defined as an Euler product
\[
    L(Y,s):=\prod_\p L_\p(Y,s),
\]
where $\p$ ranges over the prime ideals of $K$. The local $L$-factor
$L_\p(Y,s)$ is defined as follows. Choose a decomposition group
$D_\p\subset\Gal(\Kb/K)$ of $\p$. Let $I_\p\subset D_\p$ denote the inertia
subgroup and let $\sigma_\p\in D_\p$ be an arithmetic Frobenius element (i.e.\
$\sigma_\p(\alpha)\equiv\alpha^{{\rm N}\p} \pmod{\p}$). Then
\[
      L_\p(Y,s):=\det\big(1-({\rm N}\p)^{-s}\sigma_\p^{-1}|V^{I_\p}\big)^{-1},
\]
where 
\[
     V:=H^1_\et(Y\otimes_K\Kb,\Q_\ell)
\]
is the first \'etale cohomology group of $Y$ (for some auxiliary prime
$\ell$ distinct from the residue characteristic of $\p$).

Another arithmetic invariant of $Y$ closely related to $L(Y,s)$ is the
{\em conductor of the $L$-function}. Similar to $L(Y,s)$, it is
defined as a product over local factors (times a power of the
discriminant $\delta_K$ of $K$):
\[
      N:=\delta_K^{2g}\cdot\prod_\p ({\rm N}\p)^{f_\p},
\]
where $f_\p$ is a nonnegative integer called the {\em exponent of
  conductor} at $\p$. The integer $f_\p$ measures the {\em
  ramification} of the Galois module $V$ at the prime $\p$. See
\S~\ref{etale1} or \cite{SerreZeta}, \S~ 2 for a precise definition.

Many spectacular conjectures and theorems concern these $L$-functions.  For
instance, it is conjectured that $L(Y,s)$ has a meromorphic continuation to
the entire complex plane, and a functional equation of the form
\begin{equation} \label{funceq}
      \Lambda(Y,s) = \pm\Lambda(Y,2-s),
\end{equation}
where 
\[
   \Lambda(Y,s):=N^{s/2}(2\pi)^{-gs}\Gamma(s)^gL(Y,s).
\]
This conjecture can be proved for certain special curves related to
automorphic forms (like modular curves) and, as a consequence of the
Taniyama--Shimura conjecture, for elliptic curves over $\Q$. Besides that,
very little is known.

\subsection{}

One motivation for this paper is the question how to compute the
defining series for $L(Y,s)$ and the conductor $N$ explicitly for a
given curve $Y$. By definition, this is a local problem at each prime
ideal $\p$. So we fix $\p$ and aim at computing $L_\p(Y,s)$ and
$f_\p$. Note that the residue field of $\p$ is the finite field $\F_q$
with $q={\rm\mathop N}(\p)$ elements. To study this problem, we
construct suitable $\OO_K$-{\em models} of $Y$. Recall that an
$\OO_K$-{\em model} of $Y$ is a flat and proper $\OO_K$-scheme $\cY$
with generic fiber $Y$.

Assume first that $Y$ has {\em good reduction} at $\p$. This means
that there exists an $\OO_K$-model $\cY$ whose special fiber
$\Yb=\Yb_\p$  at $\p$ is a smooth $\F_q$-scheme. Standard
theorems in \'etale cohomology show that the action of $\Gal(\Kb/K)$
on $V=H^1_\et(Y_{\Kb},\Q_\ell)$ is unramified at $\p$ (i.e.\ $I_\p$
acts trivially) and therefore the exponent of conductor vanishes,
$f_\p=0$. Furthermore, the local $L$-factor $L_\p(Y,s)$ is equal to
the inverse of the denominator of the zeta function of $\Yb$, i.e.
\[
      Z(\Yb,q^{-s}) = \frac{L_\p(Y,s)^{-1}}{(1-q^{-s})(1-q^{1-s})},
\]
where 
\[
   Z(\Yb,T):=\exp\bigg(\sum_{n\geq 1} |\Yb(\F_{q^n})|\cdot \frac{T^n}{n}\bigg).
\]
To compute $L_\p(Y,s)$ for small prime ideals we simply need to
count the number of $\F_{q^n}$-rational points on $\Yb$, for
$n=1,\ldots,g$. 

If $Y$ has bad reduction it is much harder to compute $L_\p(Y,s)$ and
$f_\p$. To our knowledge, there are essentially three ways to
proceed.
\begin{enumerate}
\item
  Compute a {\em regular model} of $Y$ at $\p$. 
\item
  Compute the {\em semistable reduction} of $Y$ at $\p$. 
\item
  Guess the local $L$-factors at {\em all} primes of bad reduction, and then
  verify this guess via the functional equation for $L(Y,s)$. 
\end{enumerate}
All three methods have certain advantages and drawbacks, and it is often a
combination of them which works best. In this paper we would like to advertise
method (2), by demonstrating its simplicity and usefulness in a large class of
examples (superelliptic curves). 

\subsection{}

Before we go into more details of methods (1) and (2), let us briefly describe
method (3). Let $\p_1,\ldots,\p_r$ be the prime ideals of the number field $K$
where $Y$ has bad reduction. One can show the following.
\begin{itemize}
\item For $i=1,\ldots,r$ there are only finitely many possible choices for the
  local $L$-factor $L_{\p_i}(Y,s)$ and the exponent $f_{\p_i}$. In fact, the
  set of all choices depends only on the norm $q_i={\rm N}\p_i$ and the genus
  $g$.
\item
  There is at most a unique choice for the conductor $N$ and the local
  $L$-factors $L_{\p_i}(Y,s)$ at the bad primes $\p_i$ such that the
  $L$-function 
  \[
        L(Y,s):=\prod_\p L_\p(Y,s)
  \]
  satisfies the functional equation \eqref{funceq}. 
\end{itemize}
This suggests the following strategy to determine $L(Y,s)$.
\begin{itemize}
\item
  Guess the conductor $N=\prod_i q_i^{f_i}$ and the local $L$-factors
  $L_{\p_i}(Y,s)$ at the bad primes $\p_i$.
\item Compute $L_\p(Y,s)$ for all good primes $\p$ with ${\rm N}\p\leq
  C$ for some sufficiently large constant $C$. The constant $C$ should be
  chosen large enough, so that knowing $L_\p(Y, s)$ for all primes
  with ${\rm N}\p\leq C$ yields a sufficiently good numerical
  approximation of the $L$-function. If $C$ is not too large,
  computing $L_\p(Y,s)$ for all such good primes can be done
  efficiently by simple point counting.
\item
  Check numerically whether $L(Y,s):=\prod_\p L_\p(Y,s)$ satisfies the
  functional equation \eqref{funceq}. By \cite{Dokchitser04}, we need to
  choose $C\sim N^{1/2}$. 
\end{itemize}
In practice, this can be done if $N\sim 10^{15}$. See
e.g.\ \cite{DokchitserdeJeuZagier}. 

An obvious drawback of this method is that one can never prove that the guess
one has made is correct.

\subsection{Regular models}\label{regsec}

We now describe the first method. Fix a prime ideal $\p$ of $K$. Since the
local $L$-factor $L_\p(Y,s)$ and the exponent $f_\p$ only depend on the base
change of $Y$ to the completion $\hat{K}_\p$, we may and will from now on
assume that $K$ is a finite extension of $\Q_p$. We use the notation
$L(Y/K,s)$ and $f_{Y/K}$ to denote the local $L$-factor and the exponent of
conductor. We write $\FF_K$ for the residue
field of $K$, which is a finite field of characteristic $p$.

We may assume that $Y$ has bad reduction. By resolution of
  singularities of two-dimensional schemes, there exists a {\em
  regular model} $\cY^\reg$, i.e.\ a flat and proper $\OO_K$-model of
$Y$ which is regular. Since we assume $g\geq 2$ we may also assume
that $\cY^\reg$ is the minimal regular model. Let $\Yb^{\rm reg}$
denote the special fiber of $\cY^\reg$.  Under an additional
(relatively mild) assumption, it is still true that $L(Y/K,s)$ is the
inverse of the denominator of the zeta function of the special fiber
$\Yb^{\rm reg}$ of $\cY^{\rm reg}$ as in the smooth case (see
Proposition \ref{ssredprop} below).  Therefore, $L(Y/K,s)$ can be
computed from $\Yb^\reg$ by point counting.

By a result of Saito (\cite{Saito88}) it should also be possible to compute
$f_{Y/K}$ from $\cY^{\rm reg}$. For curves of genus $2$ this is achieved in
\cite{Liu94}, and these methods probably extend to arbitrary hyperelliptic
curves (see \cite{Liu96}). We are not aware of any attempt to explicitly
compute $f_{Y/K}$ for nonhyperelliptic curves, using regular models.

Finding a regular model $\cY^\reg$ can be computationally
challenging. The computer algebra system {\sc Magma} has a build-in function
to compute regular models of curves of genus $g\geq 2$, but it seems
that there are still many restrictions on the types of curves for
which it works. A similar function which should overcome these
limitations is being prepared in {\sc Singular}.

\subsection{Semistable reduction}\label{sssec}
   We now describe the second
 method. For precise definitions and more details we refer to
 \S~\ref{etale3}. Since we assume that $g\geq 2$, the curve
 $Y_L:=Y\otimes_KL$ admits a stable model ${\mathcal Y}^{\rm stab}$
 over a finite extension $L/K$. The stable model ${\mathcal Y}^{\rm
   stab}$ is minimal with the property that its special fiber
 $\Yb^{\rm stab}$ has at most ordinary double points as
 singularities. However, $\cY^{\rm stab}$ need not be regular.

 We may assume that $L/K$ is Galois.  The Galois group
 $\Gamma:=\Gal(L/K)$ has a natural { semilinear} action on $\cY^{\rm
   stab}$. Restricting this action to the special fiber we obtain a
 natural, semilinear action of $\Gamma$ on the special fiber $\Yb^{\rm
   stab}$ of ${\mathcal Y}^{\rm stab}$.  The quotient scheme $\Zb^{\rm
   inert}:=\Yb^{\rm stab}/\Gamma$ is a semistable curve over the
 residue field $\FF_K$ of $K$. We call it the {\em inertial reduction} of
 $Y$. The following result is certainly known
 to experts, but not so easy to find in the literature.

\begin{theorem} \label{introthm} 
  The stable reduction $\Yb^{\rm stab}$,
  together with its natural $\Gamma$-action, determines the local $L$-factor
  $L(Y/K,s)$ and the exponent of conductor $f_{Y/K}$. In particular:
  \begin{enumerate}
  \item The local $L$-factor $L(Y/K,s)$ is the inverse of the denominator of
    the zeta function of  $\Zb^{\rm inert}$
    (which may be computed by point counting).
  \item
    If, moreover, $Y$ has semistable reduction over a tamely ramified
    extension of $K$ then 
    \[
           f_{Y/K} = 2g(Y) - \dim H^1_\et(\Zb_{k}^{\rm inert},\QQ_{\ell}).
    \]
Here $k$ is the algebraic closure of $\FF_K$.
  \end{enumerate}
\end{theorem}   

  The first statement of Theorem
\ref{introthm}.(1) follows from Corollary \ref{ssredcor1}. That
corollary shows that one may use somewhat more general models of $Y$.  The
computational aspects are discussed in \S~\ref{cohomology}. Theorem
\ref{introthm}.(2) is Corollary \ref{ssredcor2}. An analogous
statement in the wild case can be found in \S~\ref{etale5}.

\subsection{}

Let us compare the two methods discussed in \S~\ref{regsec} and
\S~\ref{sssec}.  If the curve $Y$ already has semistable reduction,
the minimal regular model of $Y$ is also semistable. In this case
there is no essential difference between the two methods. In general,
however, the two methods are quite different in nature.

From the theoretical point of view one may consider the method of stable
reduction as `better' because it gives more information. For instance, unlike
the regular model, the stable model is invariant under base change of the
curve $Y$ to any finite extension $K'/K$. Therefore, once the stable reduction
of $Y$ has been computed, we can directly compute $L(Y'/K',s)$ and
$f_{Y'/K'}$, where $Y':=Y\otimes_KK'$. 

From a computational point of view it may seem to be a lot easier to
find a regular model. After all, to compute a semistable model is
essentially equivalent to computing a regular model over a larger
field $L$ {\em and} to find the correct extension $L/K$ in the first
place. However, one goal of the present paper is to show that, at
least for special classes of
curves, it is actually rather easy to determine the stable reduction,
even though the reduction behavior can be arbitrarily complicated.

\subsection{Superelliptic curves}

 We consider {\em
  superelliptic curves}, i.e.\ curves $Y$
given by an equation of the form
\[
     y^n=f(x),
\]
where $n$ is a positive integer and $f(x)$ is a rational function over a
$p$-adic number field $K$. The additional and crucial condition we impose is
that the exponent $n$ must be prime to the residue characteristic $p$ of $K$.

Let $L_0/K$ be the splitting field of $f(x)$, i.e.\ the smallest
extension of $K$ over which all poles and zeros of $f(x)$ become
rational. Our main result in \S~\ref{kummer} says that $Y$ has
semistable reduction over an explicit and at most tamely ramified
extension $L/L_0$. Moreover, the stable reduction $\Yb^{\rm stab}$,
together with the natural action of $\Gamma=\Gal(L/K)$, can be
described easily and in a purely combinatorial manner. The only part
which may be computationally difficult is the analysis of the
extension $L_0/K$. Indeed, by choosing $f(x)$ appropriately we can
make this extension as large and as complicated as we want. However,
it is possible to construct examples where the computation of the
stable reduction is still rather easy, but the standard algorithms for
computing a regular model fail.

Starting from the description of the stable reduction, we  give an
explicit procedure to determine an equation for the inertial reduction
$\Zb^{\rm inert}=\Yb^{\rm stab}/\Gamma$ in \S~\ref{mono}. This equation
can then be used to compute the local $L$-factor of $Y$ and the exponent of
conductor $f_{Y/K}$, via Theorem \ref{introthm}.

We remark that our description of the stable reduction of
superelliptic curves is based on a very special case of more general
results on {\em admissible reduction} for covers of curves. These
results are well known to experts. One of the goals of the present paper
is to make these results more widely known and to demonstrate their
usefulness for explicit computations.  In a subsequent paper, we will
present a software implementation of our results.

\section{Stable and inertial reduction}  \label{etale}

In this section we prove Theorem \ref{introthm}.

\subsection{} \label{etale1}

Let $p$ be a prime number and $K$ a finite extension of $\Q_p$. The
residue field of $K$ is a finite field, which we
denote  by $\FF_K$. The residue field of a finite extension $L/K$ is
denoted by $\FF_L$.

We choose an algebraic closure $\Kb$ of $K$ and write $\Gamma_K=\Gal(\Kb/K)$
for the absolute Galois group of $K$. The residue field of $K^\alg$ is denoted
by $k$; it is the algebraic closure of $\FF_K$. 

Let $K^\nr\subset\Kb$ be the maximal
unramified extension of $K$ and $I_K:=\Gal(\Kb/K^\nr)$ the inertia group of
$K$. We have a short exact sequence
\[
      1 \to I_K \to \Gamma_K \to \Gamma_{\FF_K} \to 1,
\]
where $\Gamma_{\FF_K}=\Gal(k/\FF_K)$ is the absolute Galois group of
$\FF_K$. This is the free profinite group of rank one
generated by the {\em Frobenius element} $\sigma_q$,  defined by
$\sigma_q(\alpha):=\alpha^q$, where $q=|\FF_K|$.

\subsection{} \label{etale2}

Let $Y/K$ be a smooth projective and
absolutely irreducible curve over $K$. We assume that the genus $g$ of $Y$
satisfies $g\geq 2$.  We fix an auxiliary prime $\ell\neq p$.
As explained in the introduction, we are interested in computing
certain invariants of the  natural action
of $\Gamma_K$ on the \'etale cohomology group
\[
   V=H^1_\et(Y_{\Kb},\Q_\ell) := 
    \big(\varprojlim_n H^1_\et(Y_{\Kb},\Z/\ell^n)\big)\otimes \Q_\ell.
\]
The {\em local
$L$-factor} is the function $L(Y/K,s):=P_1(Y/K,q^{-s})^{-1}$, where
\[
        P_1(Y/K,T):=\det(1-\sigma_q^{-1}\cdot T\mid V^{I_K}).
\]
The {\em exponent of conductor} is defined as the integer
\begin{equation} \label{conductoreq}
     f=f_{Y/K}=\epsilon+\delta,
\end{equation}
where 
\begin{equation}\label{epsiloneq}
      \epsilon := \dim V - \dim V^{I_K}
\end{equation}
is the codimension of the $I_K$-invariant subspace and $\delta$ is the
 {\em Swan conductor} of $V$ (see \cite{SerreZeta} \S~2, or
\cite{Wiese08}, \S~3.1).

The invariant $f_{Y/K}$ depends only on the $I_K$-action on $V$, and vanishes
if the $I_K$-action is trivial (i.e.\ if $V$ is {\em unramified}). In general
it gives a measure of `how bad' the ramification of $V$ is.

\subsection{} \label{etale3}

A theorem of Deligne and Mumford (\cite{DeligneMumford69}) states the
existence of a finite extension $L/K$ such that the curve
$Y_{L}=Y\otimes_K L$ has {\em semistable reduction}. This means that there
exists a flat and proper $\OO_{L}$-model $\cY$ of $Y_{L}$ whose special
fiber $\Yb$ is reduced and has at most ordinary double points as
singularities. The model $\cY$ is not unique, but the assumption
$g\geq 2$ implies that there is a minimal semistable model $\cY^{\rm
  stab}$, called the {\em stable model} of $Y_L$. The special fiber
$\Yb^{\rm stab} $ of $\cY^{\rm stab}$ is called the {\em stable
  reduction} of $Y_L$. It is a stable curve over the residue field
$\FF_L$, uniquely determined by the $K$-curve $Y$ and the extension
$L/K$. The dependence on $L$ is very mild: if $L'/L$ is a further
finite extension then the stable reduction of $Y$ corresponding to the
extension $L'/K$ is just the base change of $\Yb^{\rm stab}$ to the
residue field of $L'$.

If $\cY$ is an arbitrary semistable model of $Y_L$,  there exists a unique
$\OO_L$-morphism $c:\cY\to\cY^{\rm stab}$ which is the identity on the generic
fiber. The morphism $c$ contracts the instable components of the special fiber
of $\cY$ and is an isomorphism everywhere else. Here an irreducible component
$C$ of the special fiber of $\cY$ is called {\em instable} if $C$ is smooth of
genus zero and intersects the rest of the special fiber in at most two points.

After replacing $L$ by a suitable finite extension we may and will
henceforth assume that $L/K$ is a Galois extension. We also choose an
embedding $L\subset K^{\rm alg}$. Then the absolute Galois group
$\Gamma_K$ acts naturally on $Y_L$ via its finite quotient
$\Gamma:=\Gal(L/K)$. Let $I\lhd\Gamma$ denote the inertia
subgroup, i.e.\ the image of $I_K$ in $\Gamma$. Note that the action
of $\Gamma$ on $Y_L$ is only $L/K$-semilinear, but its restriction to
$I$ is $L$-linear.

\begin{definition} \label{quasistabledef}
  A semistable $\OO_L$-model $\cY$ of $Y_L$ is called {\em
    quasi-stable} if the tautological action of $\Gamma$ on $Y_L$ extends
  to an action on $\cY$.
\end{definition}  

The uniqueness of the stable model shows that it is quasi-stable. For our
purposes it is more convenient to work with an arbitrary quasi-stable model
$\cY$. Let $\Yb$ denote the special fiber of $\cY$. Restricting the canonical
$\Gamma$-action on $\cY$ to $\Yb$ yields a canonical action of $\Gamma$ on
$\Yb$. This action is again semilinear, meaning that the structure map
$\Yb\to\Spec \FF_L$ is $\Gamma$-equivariant. However, the action of the
inertia group $I$ on $\Yb$ is $\FF_L$-linear.

We let $\Zb:=\Yb/\Gamma$ denote the quotient scheme. It has a natural
structure of an $\FF_K$-scheme, and as such we have
$\Zb_{\FF_L}:=\Zb\otimes_{\FF_K}\FF_L=\Yb/I$. Since the quotient of a
semistable curve by a finite group of geometric automorphisms is
semistable, it follows that $\Zb\otimes_{\FF_K}\FF_L$ is a semistable
curve over $\FF_L$. We conclude that $\Zb$ is a semistable curve over
$\FF_K$. We denote by $\Zb_k:=\Zb\otimes_{\FF_K} k$ the base change of
$\Zb$ to the algebraic closure $k$ of $\FF_K$.

\begin{definition} \label{inertialreddef}
  The $\FF_K$-curve $\Zb=\Yb/\Gamma$ is called the {\em inertial
    reduction} of $Y$, corresponding to the quasi-stable model $\cY$.
  \end{definition}

\begin{remark} 
In \S~\ref{sssec} we considered the inertial reduction $\Zb^{\rm
  inert}$ corresponding to the stable model $\cY^{\rm stab}$.  It is canonically
associated with the $K$-curve $Y$ and does not depend on the choice of
the Galois extension $L/K$. 

An arbitrary quasi-stable model $\cY$ admits an contraction map
$c:\cY\to\cY^{\rm inert}$, which is $\Gamma$-equivariant. The inertial
reduction $\Zb$ corresponding to $\cY$ admits therefore a map
$\Zb\to\Zb^{\rm inert}$ contracting the components of $\Zb$ which are
the image of the instable components of $\Yb$. The image of a
stable component of $\Yb$ may be an instable component of $\Zb$. So in
general, $\Zb^{\rm inert}$ is not a stable curve.
\end{remark}

The following theorem is the main result of this section.

\begin{theorem} \label{ssredthm2} Let $\Zb$ be the inertial reduction of $Y$
  corresponding to some quasi-stable model $\cY$.  We have a natural,
  $\Gamma_K$-equivariant isomorphism
  \[
       H^1_\et(Y_{\Kb},\Q_\ell)^{I_K} \cong H^1_\et(\Zb_k,\Q_\ell).
  \]
\end{theorem}

\begin{corollary} \label{ssredcor1}
  In the situation of Theorem \ref{ssredthm2}, the local
  $L$-factor $L(Y/K,s)$ is equal to the numerator of the local zeta
  function of $\Zb$, i.e.\
  \[
       L(Y/K,s)=P_1(\Zb,q^{-s})^{-1},
  \]
  where 
  \[
      P_1(\Zb,T):=\det\big(1-{\rm Frob}_q\cdot T|H^1_\et(\Zb_k,\Q_\ell)\big)
  \]
  and ${\rm Frob}_q:\Zb\to \Zb$ is the relative $q$-Frobenius
  endomorphism and $q=|\FF_K|$.
\end{corollary}

\begin{proof} 
  The action of $\Gamma_K$ on $H^1_\et(\Zb_k,\Q_\ell)$ factors through the
  quotient $\Gamma_K\to\Gamma_{\FF_K}$. The resulting $\Gamma_{\FF_K}$-action
  is the same as the action induced by the identification $\Zb_k=\Zb\otimes
  k$. It follows that the action of an arithmetic Frobenius element
  $\sigma_q\in\Gamma_K$ on $H^1_\et(\Zb_k,\QQ_\ell)$ is induced by the map
  ${\rm Id}_{\Zb}\otimes\sigma_q$. But the composition $({\rm
    Id}_{\Zb}\otimes\sigma_q)\circ({\rm Frob}_q\otimes{\rm Id}_{k})$ is equal
  to the {\em absolute $q$-Frobenius} of $\Zb_k$. Since the absolute Frobenius
  induces the identity on \'etale cohomology, it follows that ${\rm
    Frob}_q=\sigma_q^{-1}$ on $H^1_\et(\Zb_k,\QQ_\ell)$. (This is a standard
  argument, see e.g.\ \cite{DeligneBourbaki355}, Proposition 4.8 (ii) or
  \cite{Chenevert04}.)  The claim is now a consequence of Theorem
  \ref{ssredthm2} and the definition of $L(Y/K,s)$.
\end{proof}

Corollary \ref{ssredcor1} implies that we can compute the local
$L$-factor $L(Y/K,s)$ from the explicit knowledge of the inertial
reduction $\Zb$.  In a special case, this is also enough to determine
the exponent of conductor $f_{Y/K}$. The computation of $f_{L/K}$
without the tameness assumption is described in \S~\ref{etale5}.

\begin{corollary} \label{ssredcor2} 
  Assume that  the extension $L/K$ in  Theorem \ref{ssredthm2}
  is at most tamely ramified. Then  
  \[
        f_{Y/K}=2g(Y)-\dim H^1_\et(\Zb_k,\QQ_\ell).
  \]
\end{corollary}

\begin{proof}
If the extension $L/K$ is at most tamely ramified, the action of
$\Gamma_K$ on $H^1_\et(Y_{K^\alg},\QQ_\ell)$ is tame. The definition
of the Swan conductor implies that $\delta=0$ in
\eqref{conductoreq}. The claim is now a direct consequence of Theorem
\ref{ssredthm2} and the definition of the conductor $f_{Y/K}$ in
(\ref{conductoreq}).
\end{proof}

\subsection{}\label{cohomology} 

Corollary \ref{ssredcor1} reduces the calculation of the local
$L$-factor to the calculation of the relative Frobenius endomorphism
on the \'etale cohomology of the semistable curve $\Zb$.  The
following well-known lemma describes this action.

In this subsection we let $\Zb/\FF_K$ be an arbitrary semistable curve
defined over the finite field $\FF_K$. Let $k$ be the algebraic
closure of $\FF_K$ and $\Zb_k$ the base change to $k$. Denote by
$\pi:\Zb_k^{(0)}\to \Zb_k$ the normalization. Then $\Zb_k^{(0)}$ is
the disjoint union of its irreducible components, which we denote by
$(\Zb_j)_{j\in J}$.  These correspond to the irreducible components
of $\Zb_k$. The components $\Zb_j$ are smooth projective curves.  The
absolute Galois group $\Gamma_{\FF_K}$ of $\FF_K$ naturally acts on
the set of irreducible components. We denote the permutation character
of this action by $\chi_{\rm comp}.$

Let $\xi\in \Zb_k$ be a singular point. Then $\pi^{-1}(\xi)\subset
\Zb_k^{(0)}$ consists of two points. We define a $1$-dimensional character
$\varepsilon_\xi$ on the stabilizer $\Gamma_{\FF_K}(\xi)\subset\Gamma_{\FF_K}$
of $\xi$ as follows.  If the two points in $\pi^{-1}(\xi)$ are permuted by
$\Gamma_{\FF_K}(\xi)$, then $\epsilon_\xi$ is the unique character of order
two. Otherwise, $\varepsilon_\xi={\boldsymbol 1}$ is the trivial character.
Denote by $\chi_\xi$ the character of the induced representation
\[
   \Ind_{\Gamma_{\FF_K}(\xi)}^{\Gamma_{\FF_K}}\varepsilon_\xi.
\]
In the case that $\varepsilon_\xi=1$ this is just the character of
the permutation representation of the orbit of $\xi$. Define
\[
   \chi_{\rm sing}=\sum_{\xi\in \Zb_k^{\rm sing}}\chi_\xi.
\]

We denote by $\Delta_{\Zb_k}$ the graph of components of $\Zb_k$.

\begin{lemma}\label{cohomologylem}
Let $\Zb/\FF_K$ be a semistable curve and $\ell$ a prime with $\ell\nmid q$. 
\begin{enumerate}
\item We have a decomposition
\[
   H^1_\et(\Zb_k,\Q_\ell)= \oplus_{j\in
     J}H^1_\et(\Zb_j, \QQ_\ell)\oplus H^1(\Delta_{\Zb_k})
\]
as $\Gamma_{\FF_K}$-representation. 
\item The character of $H^1(\Delta_{\Zb_k})$ as
  $\Gamma_{\FF_K}$-representation is $1+\chi_{\rm sing}-\chi_{\rm
  comp}$.
\end{enumerate}
\end{lemma}

\begin{proof}
As before, we let $\pi:\Zb_k^{(0)}\to \Zb_k$ be the normalization.
We have a short exact sequence
\[
     0\to  \Q_\ell \to \pi_\ast(\Q_\ell)\to Q\to 0
\]
of sheaves on $\Zb_k$, where $Q:=\pi_\ast(\Q_\ell)/\Q_\ell$ is a
skyscraper sheaf with support in the singular points. This induces
\begin{gather*}
0\to H^0_\et(\Zb_k, \Q_\ell)\to H^0_\et(\Zb_k,\pi_\ast(\Q_\ell))\to
H^0_\et(\Zb_k,Q)\to\\ H^1_\et(\Zb_k,\Q_\ell)\to
H^1_\et(\Zb_k,\pi_\ast(\QQ_\ell)\to 0.
\end{gather*}

Identifying $H^0_\et(\Zb_k,\pi_\ast(\Q_\ell))$ with $\QQ_\ell^J$, we find that
the kernel of the map $H^1_\et(\Zb_k,\Q_\ell)\to
H^1_\et(\Zb_k,\pi_\ast(\QQ_\ell))$ equals $H^0_\et(\Zb_k, \Q_\ell)\oplus
H^0_\et(\Zb_k,Q)/\QQ_\ell^J$ as $\Gamma_{\FF_K}$-representation. It is easy to
see that the character of $H_\et^0(\Zb_k,Q)$ is equal to $\chi_{\rm
  sing}$. This proves (2). Since
$H^1_\et(\Zb_k,\pi_\ast(\QQ_\ell)=\oplus_{j\in J}H^1_\et(\Zb_j,\Q_\ell)$, (1)
follows as well.
\end{proof}

The irreducible components of $\Zb$ are in general not absolutely
irreducible. An irreducible component $\Zb_{[j]}$ of $\Zb$ decomposes
in $\Zb_k$ as a finite union of absolutely irreducible curves, which
form an orbit under $\Gamma_{\FF_K}$. Let $\Zb_j$ be a representative
of the orbit.  Let $\Gamma_j\subset \Gamma_{\FF_K}$ be the stabilizer
of $\Zb_j$ and $\FF_{q_j}=k^{\Gamma_j}.$ We may identify $\Zb_{[j]}$
and $\Zb_j/\Gamma_j$ as absolute schemes. The natural
$\FF_K$-structure of $\Zb_{[j]}$ (which is missing from
$\Zb_j/\Gamma_j$) is given by
\[
\Zb_j/\Gamma_j\to \Spec(\FF_{q_j})\to \Spec(\FF_K).
\]
With this interpretation, the contribution of $\Zb_{[j]}$ to the local
zeta function in Corollary \ref{ssredcor1} can be computed explicitly
using point counting. We refer to \S~\ref{exa23} for an example where
$\FF_{q_j}\neq \FF_K$.

Summarizing, we see that to compute the local L-factor it suffices to
describe the irreducible components of the normalization $\Zb^{(0)}$
of $\Zb$ using equations over $\FF_K$, together with the inverse image
$\Zb^{(1)}\subset\Zb^{(0)}$ of the singular locus of $\Zb$. In the
situation of Corollary \ref{ssredcor2} the same information also
yields the exponents of conductor. In the general case we need
somewhat more information (Theorem \ref{ssredthm3} below), which may
be calculated in an equally explicit way.
 For superelliptic curves this
will be done in \S~\ref{mono}.

\subsection{} \label{etale4}

The proof of Theorem \ref{ssredthm2} relies on the following
(well-known) proposition.

\begin{proposition}\label{ssredprop}
  Let $K$ be a henselian local field. Let $k$ denote the algebraic
  closure of the residue field of $K$. Let $Y$ be a smooth projective
  curve over $K$ and $\cY$ be an $\OO_K$-model of $Y$ which is
  semistable or regular.  If $\cY$ is regular we assume moreover that
  the gcd of the multiplicities of the components of the special fiber
  $\Yb$ of $\cY$ is one.  Then the cospecialization map induces an
  isomorphism
  \[
       H^1_\et(Y_{\Kb},\Q_\ell)^{I_K} \cong H^1_\et(\Yb_k,\Q_\ell).
  \]
\end{proposition}

\begin{proof} By \cite{MilneEC}, Corollary 4.18, we have
isomorphisms
\begin{equation} \label{ssredpropeq1}
     H^1_\et(Y_{\Kb},\Q_\ell(1)) \cong V_\ell({\rm Pic}^0(Y)), \quad
     H^1_\et(\Yb_k,\Q_\ell(1)) \cong V_\ell({\rm Pic}^0(\Yb)),
\end{equation}
where $V_\ell(\,\cdot\,)$ denotes the rational $\ell$-adic Tate module.  

Let $\J$ denote the N\'eron model of the Jacobian of $Y$ and $\Jb^0$ the
connected component of its special fiber. Then by \cite{SGA7I}, 6.4 (see also
\cite{SerreTate68}, Lemma 2) we have
\begin{equation} \label{ssredpropeq2}
    V_\ell({\rm Pic}^0(Y))^{I_K} \cong V_\ell(\Jb^0).
\end{equation}
Under the conditions imposed on $\cY$ we have an
isomorphism
\begin{equation} \label{ssredpropeq3}
    \Jb^0 \cong {\rm Pic}^0(\Yb)
\end{equation}
by \cite{BLR}, Theorem 9.5.4 and Corollary 9.7.2. The proposition follows by
combining \eqref{ssredpropeq1}, \eqref{ssredpropeq2} and \eqref{ssredpropeq3}.
\end{proof}

\begin{proof} We prove Theorem \ref{ssredthm2}.
  Let $L/K$ be a finite Galois extension over which $Y$ has
semistable reduction. Let $\cY$ be a quasi-stable model of $Y_L$ and
$\Yb$ its special fiber.  Proposition \ref{ssredprop} yields an isomorphism
\[
     H^1_\et(Y_{\Kb},\Q_\ell)^{I_L} \cong H^1_\et(\Yb_k,\Q_\ell)
\]
which is canonical, and therefore $\Gamma_K$-invariant.
Taking $I_K$-invariants and using the Hochschild--Serre spectral sequence
(\cite{MilneEC}, III.2.20), we conclude that 
\[
    H^1_\et(Y_{\Kb},\Q_\ell)^{I_K} \cong H^1_\et(\Yb_k,\Q_\ell)^{I_K}
        \cong H^1_\et(\Yb_k/I_K,\Q_\ell).
\]
Since $\Yb_k/I_K=\Zb_k$, Theorem \ref{ssredthm2} follows.
\end{proof}

\subsection{} \label{etale5}

We give a formula for the exponent of conductor
$f_{Y/K}$ in terms of the stable reduction $\Yb$ that works in
general, i.e.\ without the tameness assumption of Corollary
\ref{ssredcor2}. 

The exponent of conductor is defined in (\ref{conductoreq}) as
$f_{Y/K} = \epsilon +\delta$. Theorem \ref{ssredthm2} and
(\ref{epsiloneq}) imply that
\begin{equation}\label{epsiloneq2}
  \epsilon = 2g_Y-\dim H^1_\et(\Zb_k,\QQ_\ell).
\end{equation}
Therefore $\epsilon$ may be computed from the
inertial reduction $\Zb$.

The following result expresses the Swan conductor $\delta$ in terms of
the special fiber $\Yb$ of a quasi-stable model $\cY$.  Let
$(\Gamma_i)_{i\geq 0}$ be the filtration of $\Gamma=\Gal(L/K)$ by
higher ramification groups. Then $\Gamma_0=I$ is the inertia group and
$\Gamma_1=P$ its Sylow $p$-subgroup (\cite{SerreCL}, Chapter
4). Moreover, $\Gamma_i=1$ for $i\gg 0$. Let $\Yb_i:=\Yb/\Gamma_i$ be
the quotient curve. Then $\Yb_0=\Yb/I=\Zb_{\FF_L}$ and $\Yb_i=\Yb$ for
$i\gg 0$.

\begin{theorem} \label{ssredthm3}
  The Swan conductor is
  \[
      \delta=\sum_{i=1}^\infty\, \frac{|\Gamma_i|}{|\Gamma_0|}
              \cdot(2g_Y-2g_{\Yb_i}).
  \]
  Here $g_{\Yb_i}$ denotes the arithmetic genus of $\Yb_i$.
\end{theorem} 

\begin{proof} Let $I_K^w\subset\Gamma_K$ denote the wild inertia
subgroup. The image of $I_K^w$ in the finite quotient
$\Gamma=\Gal(L/K)$ is equal to $\Gamma_1$. It follows from
\cite{AbbesLuminy}, Theorem 1.5, that the action of $I_K^w$ on
$V=H^1_\et(Y_{K^\alg},\QQ_\ell)$ factors over the
$\Gamma_1$-action. (Note that this is not true for the action of the
full inertia group $I_K$.) To compute $\delta$ we may therefore use
the {\em Hilbert formula} of \cite{Ogg67}, page $3$, which says that
\begin{equation} \label{deltaeq1}
  \delta=\sum_{i=1}^\infty\; \frac{|\Gamma_i|}{|\Gamma_0|}\cdot
                  \dim_{\QQ_\ell} V/V^{\Gamma_i}.
\end{equation}
Although {\em loc.cit} is an expression for
the Swan conductor of the mod-$\ell$-representation
$\bar{V}=H^1_\et(Y_{K^\alg},\FF_\ell)$, we can use the same formula for $V$ as
well. This follows from \cite{Wiese08}, Proposition 3.1.42. To finish the
proof it remains to show that
\begin{equation} \label{Gammaieq}
    \dim_{\QQ_\ell} V^{\Gamma_i} =  2g_{\Yb_i}
\end{equation}
for $i\geq 1$. Note again that \eqref{Gammaieq} does not hold for $i=0$: by
Theorem \ref{ssredthm2} we have $V^{\Gamma_0}=H^1_\et(\Zb_k,\QQ_\ell)$, and
the dimension of this space is equal to $2g_{\Zb}$ only if the graph of
components of $\Zb$ is a tree. 

The results of \cite{AbbesLuminy}, \S~3 imply that $V$ decomposes,
as a $\Gamma_1$-module, into the direct sum
\begin{equation} \label{Gammaieq2}
    V=H^1_\et(\Yb^{(0)})\oplus H_1(\Delta_{\Yb})\oplus H^1(\Delta_{\Yb}),
\end{equation}
where $\Yb^{(0)}$ is the normalization of $\Yb$, $\Delta_{\Yb}$ is the graph
of components of $\Yb$ and $H_1(\Delta_{\Yb})$ (resp.\ $H^1(\Delta_{\Yb})$)
denotes the (co)homology of $\Delta_{\Yb}$ with $\QQ_\ell$-coefficients.
Using the Hochschild--Serre spectral
sequence, it follows from (\ref{Gammaieq2}) that
\begin{equation} \label{Gammaieq3}
    V^{\Gamma_i}=H^1_\et(\Yb^{(0)}_i)\oplus H_1(\Delta_{\Yb_i})
          \oplus H^1(\Delta_{\Yb_i}),
\end{equation}
for $i\geq 1$. The dimension of the right-hand side of
\eqref{Gammaieq3} is equal to $2g_{\Yb_i}$, proving
\eqref{Gammaieq}. The theorem follows.  \end{proof}

\begin{remark}
The results of this section yield the following ``trivial'' upper
bound for the exponent of conductor, which is easily computed
in the case that the ramification of the extension $L/K$ is known.

If $L/K$ is at most tamely ramified we have already seen that
$\delta=0$, hence we have that $f_{Y/K}=\epsilon\leq 2g(Y)$. 

Suppose that $L/K$ is wildly ramified. Let $h$ be the last jump in the
filtration of higher ramification groups, i.e.\ $h=i$ is maximal with
$\Gamma_i\neq \{0\}$. Then Theorem \ref{ssredthm3} implies that 
$\delta\leq 2g(Y)h |P|/|\Gamma_0|.$
It follows that $f_{Y/K}=\epsilon+\delta\leq 2g(Y)(1+h|P|/|\Gamma_0|).$
\end{remark}

\section{Admissible covers} \label{admiss}

\subsection{} \label{admiss1}

Let $K/\Q_p$ be a $p$-adic number field as before and $\phi:Y\to
X=\P^1_K$ a finite cover over $K$. We assume that $Y$ is smooth,
absolutely irreducible and of genus $g\geq 2$.

Let $L/K$ be a finite extension over which $Y$ has semistable
reduction.  There exists a unique semistable model $\cX$ of $X_L$ such
that $\phi$ extends to a finite $\OO_L$-morphism $\cY^{\rm stab}\to\cX$
(\cite{LiuLorenzini99}). Moreover, the stable model $\cY^{\rm stab}$
is the normalization of $\cX$ inside the function field of $Y_L$. If
$\phi$ is a Galois cover with Galois group $G$, then the $G$-action on
$Y_L$ extends to $\cY^{\rm stab}$ and the quotient scheme
$\cX:=\cY^{\rm stab}/G$ has the desired property.

Our strategy for computing the stable reduction of $Y$ is to try to
reverse the process described above: we try to find a semistable model
$\cX$ of $X$ whose normalization $\cY$ with respect to $Y$ is again
semistable. In \cite{ArzdorfWewers} a general method for finding such
semistable model $\cX$ is developed. This approach has been made
algorithmic in \cite{ArzdorfDiss} for cyclic covers $\phi:Y\to \P^1_K$
of degree $p$, were $p$ is the residue characteristic.

The case that $\phi$ is a Galois cover where $p$ does not divide the
order of the Galois group $G$ is much easier than the ``wild''
case. In this case it is well known how to compute the stable
reduction of $Y$.  The main insight goes back to Harris--Mumford
(\cite{HarrisMumford82}) and is based on the notion of {\em admissible
  covers}. We describe the result in \S~\ref{admiss3}.

\subsection{} \label{admiss2}

We first need a generalization of the notion of a (semi)stable model. 

\begin{definition}
  Let $S$ be a scheme, $\cX\to S$ a semistable curve over $S$ and
  $s_1,\ldots,s_r:S\to\cX^{\rm sm}$ disjoint sections supported in the smooth
  locus of $\cX\to S$. Then $(\cX/S,s_1,\ldots,s_r)$ is called a {\em pointed
    semistable curve} over $S$ (cf.\ \cite{Knudsen83}). Since we are usually
  not interested in ordering the sections $s_i$, we write $\cD\subset\cX$ for
  the relative divisor composed of the images of the $s_i$ and call
  $(\cX,\cD)$ a {\em marked semistable curve}. The divisor $\cD\subset\cX$ is
  called a {\em marking} of $X/S$. 
\end{definition}

Let $K$ be a local field as before and $X/K$ a smooth projective curve.  Let
$D\subset X$ be a smooth relative divisor of degree $d$ over $\Spec K$.  We
say that $D$ {\em splits} over $K$ if $D$ consist of $d$ distinct $K$-rational
points. We say that the marked curve $(X,D)$ has {\em semistable reduction} if
$D$ splits and the pair $(X,D)$ extends to a marked semistable curve
$(\cX,\cD)$ over $\OO_K$. If this is the case, $(\cX,\cD)$ is called a {\em
  semistable model} of $(X,D)$.

The semistable reduction theorem extends to the marked case, as follows.

\begin{proposition} \label{markedprop} 
  Let $(X,D)$ be as above. 
  \begin{enumerate}
  \item
    There exists a finite extension $L/K$ such that $(X_L,D_L)$ has semistable
    reduction.
  \item
    Assume, moreover, that $2g(X)-2+d>0$. Then there exists a unique minimal
    semistable model $(\cX,\cD)$ (which we call the {\em stable model} of
    $(X,D)$).
  \item If $g=0$ and $D$ splits then $(X,D)$ has semistable
    reduction. 
  \item
    Assume that $g=0$, $d\geq 3$ and that $D$ splits. Let $(\Xb,\Db)$ be the
    special fiber of the stable model $(\cX,\cD)$ of $(X,D)$. Then $\Xb$ is a
    tree of projective lines. Every irreducible component $\Xb_v$ of $\Xb$ has
    at least three points which are either singular points of $\Xb$ or belong
    to the support of the divisor $\Db$. 
  \end{enumerate}
\end{proposition}

\begin{proof}
  Statements (1) and (2) follow from the Semistable Reduction Theorem
  (\S~\ref{etale3}) combined with the main result of
  \cite{Knudsen83}.

  Statements (3) and (4) are proved in \cite{GHvdP88}. In that paper one
  also finds a much more direct proof for (1) and (2) in the case that $g=0$.
\end{proof}

\subsection{} \label{admiss3}

We  return to the situation from the beginning of this section.  Let
$\phi:Y\to X=\P^1_K$ be a finite cover of the projective line, where $Y$
is smooth and absolutely irreducible over $K$. 

Let $D\subset X$ be the {\em branch locus} of $\phi$, i.e.\ the reduced closed
subscheme  exactly supporting the branch points of $\phi$. Then $D\to\Spec K$
is a finite flat morphism. Since the characteristic of $K$ is zero and $D$ is
reduced by definition, $D\to\Spec K$ is actually \'etale. The geometric points
of $D$ are exactly the branch points of $\phi_{K^\alg}$. Let $d$ denote the
degree of $D$, i.e.\ the number of branch points of $\phi_{\Kb}$. We make the
following additional assumptions on~$\phi$.
\begin{assumption} \label{phiass} \ 
\begin{itemize}
\item[(a)]
  The cover $\phi$ is {\em potentially Galois}, i.e.\ the base change
  $\phi_{K^\alg}:Y_{K^\alg}\to X_{K^\alg}$ is a Galois cover. 
\item[(b)]
   The 
  characteristic $p$ of the residue field of $K$ does not divide 
the order of the Galois group $G$ of  $\phi_{K^\alg}$.
\item[(c)]
  We have  $g(Y)\geq 2$. 
\end{itemize} 
\end{assumption}

Assumption \ref{phiass}.(c) implies that $d\geq 3$.

Let $L/K$ be a finite extension which splits $D$. Then $(X,D)$ has
semistable reduction over $L$ (Proposition \ref{markedprop}.(3)). Let
$(\cX,\cD)$ denote the stable model of $(X_L,D_L)$ and $\cY$ the
normalization of $\cX$ in the function field of $Y$.  Then $\cY$ is a
normal integral model of $Y$ over $\OO_L$.  Let
$\Yb:=\cY\otimes \FF_L$ be the special fiber and $\phib:\Yb\to \Xb$
the induced map.

An irreducible component $W$  of $\Yb$ corresponds to a discrete valuation
$\eta_W$ of the function field of $Y_L$ (since $W$ is a prime divisor on
$\cY$). Let $m_W$ denote the ramification index of $\eta_W$ in the extension
of function fields induced by $\phi$. The integer $m_W$ is called the {\em
  multiplicity} of the component $W$. (Alternatively, one can define $m_W$ as
the length of $\OO_{\cY,W}/(\pi)$, where $\OO_{\cY,W}$ is the local ring at the
generic point of $W$ and $\pi$ is a prime element of $\OO_L$.)

\begin{theorem} \label{admissthm1} Let $L/K$ and $\cY$ be as above. Assume
  that $\phi_L:Y_L\to X_L$ is a Galois cover and that $m_W=1$ for every
  irreducible component $W$ of $\Yb$.  Then $\cY$ is a quasi-stable model of
  $Y_L$. In particular, $\cY$ is semistable.
\end{theorem}

\begin{proof} 
The proof is a straightforward adaptation of the proof of
\cite{LiuLorenzini99}, Theorem 2.3 to our situation.
\end{proof}

\begin{remark} \label{admissrem2}
\begin{enumerate}
\item The quasi-stable model $\cY$ from Theorem \ref{admissthm1} is in
  general not the stable model of $Y$. Furthermore, the extension
  $L/K$ is in general not the minimal extension over which $Y$ has
  semistable reduction.
\item A key step in the proof of Theorem \ref{admissthm1} is showing
  that $\cY\to\cX$ is an { admissible cover} (see
  \cite{HarrisMumford82} or \cite{tame}).  For the purpose of the
  present paper, it suffices to know that this implies that smooth
  (resp.\ singular) points of $\Yb$ map to smooth (resp.\ singular)
  points of $\Xb$. Since the irreducible components of $\Xb$ are
  smooth (see \S~\ref{kummer2} below), it follows that the same holds
  for the irreducible components of $\Yb$.
\end{enumerate}
\end{remark}

\begin{corollary} \label{admisscor1} Let $\phi:Y\to X=\P^1_K$ be a cover
  satisfying Assumption \ref{phiass}, with branch locus $D\subset
  X$. Let $L_0/K$ be a finite extension which splits $D$ and such that
  $\phi_{L_0}$ is Galois. There exists a tamely ramified extension
  $L/L_0$ over which $Y$ has semistable reduction.
\end{corollary}

\begin{proof} Let $(\cX_0,\cD_0)$ be the stable model of the marked curve
  $(X_{L_0},D_{L_0})$ and $\cY_0$ the normalization of $\cX_0$ in
  $Y_{L_0}$. Let $e$ be the lcm of all multiplicities $m_W$, where $W$ runs
  over the irreducible components of the special fiber of $\cY_0$. It is clear
  that $e$ divides the order of the Galois group of $\phi_{L_0}$ and is
  therefore prime to $p$.

Let $L/L_0$ be a tamely ramified extension with ramification index
divisible by $e$. Let $(\cX,\cD)$ be the base change of
$(\cX_0,\cD_0)$ to $\OO_L$; this is the stable model of the marked
curve $(X_L,D_L)$. Let $\cY$ be the normalization of $\cX$ in
$Y_L$. It follows from Abhyankar's lemma (\cite{SGA1}, Expos\'e X,
Lemma 3.6, p.~297) that the multiplicities of the irreducible
components of $\Yb$ are one.  Theorem \ref{admissthm1} implies that
$\cY$ is semistable. This proves the corollary.
\end{proof}

\section{Superelliptic curves} \label{kummer}

\subsection{} \label{kummer1}

As before, $K/\Q_p$ is a finite extension. Let $\phi:Y\to X:=\P^1_K$ be the
cover of curves which is birationally determined by an equation of the form
$$
  y^n = f(x),
$$ where $f\in K[x]$ is a nonconstant polynomial in the natural
  parameter $x$ of the projective line $X=\P^1_K$ and
  $\phi(x,y)=x$. In other words, $Y$ is the smooth projective curve
  with function field $F(Y):=K(x,y\mid y^n=f(x))$. We assume that $f$
  has no nontrivial factor which is an $n$th power in $K[x]$. This
  implies that every zero of $f$ corresponds to a branch point of
  $\phi$.

Let $L_0/K$ be the splitting field of $f$ and $S\subset L_0$ the set of roots of
$f$. Then we can write
\[
    f=c\prod_{\alpha\in S}(x-\alpha)^{a_\alpha},
\]
with $c\in K^\times$ and $a_\alpha\in\N$. We impose the
following conditions on $f$ and $n$.
\begin{assumption} \label{fnass}
\begin{itemize}
\item[(a)]
  We have $\gcd(n,a_\alpha\mid \alpha\in S)=1$. 
\item[(b)]
  The exponent $n$ is $\geq 2$ and prime to $p$.
\item[(c)]
  We have $g(Y)\geq 2$.
\end{itemize}
\end{assumption}

We note that Assumption \ref{fnass} implies Assumption
\ref{phiass}. In fact, the base change of $\phi$ to $K^{\rm ur}$ is a
Galois cover with cyclic Galois group of order $n$, branched over the
roots of $f$ and possibly also over $\infty$. The ramification index
of the points of $\phi^{-1}(\infty)$ is $n/\gcd(n, \sum_{\alpha\in S} a_\alpha)$.

Our goal is to compute the stable reduction of $Y$ in terms of the data $f$
and $n$, following the procedure suggested by Theorem \ref{admissthm1}
and Corollary \ref{admisscor1}.

\subsection{} \label{kummer2}

Let $D\subset X$ be the branch divisor of $\phi$. Let $L/L_0$ be a
finite extension. Then $D$ splits over $L$, and $D_L\subset \P^1(L)$ satisfies
\[
     D_L = \begin{cases}
        S & \text{if $\sum_{\alpha\in S} a_\alpha\equiv 0 \pmod{n}$,} \\
        S\cup\{\infty\} & \text{otherwise.}
           \end{cases}
\]
Assumption \ref{fnass}.(c) implies that $|D_L|\geq 3$. Therefore the
marked curve $(X_L,D_L)$ has a stable model $(\cX,\cD)$ (Proposition
\ref{markedprop}). In the rest of this section we describe the special
fiber $(\Xb,\Db)$ of $(\cX,\cD)$ explicitly, in terms of the divisor
$D_L\subset X_L$.

We first introduce some notation. Let $\Delta=(V(\Delta),E(\Delta))$
denote the graph of components of $\Xb$. This is a finite, undirected
tree whose vertices $v\in V(\Delta)$ correspond the irreducible
components $\Xb_v\subset\Xb$. Two vertices $v_1,v_2$ are adjacent if
and only if the components $\Xb_{v_1}$ and $\Xb_{v_2}$ meet in a
(necessarily unique) singular point of $\Xb$. For an element
$\alpha\in D_L$ we denote by $\alphab\in\Db\subset\Xb$ its
specialization. We obtain a map $\psi:D_L\to V(\Delta)$ defined by 
$\alphab\in\Xb_{\psi(\alpha)}$. Proposition \ref{markedprop}.(4)
states that $(\Delta,\psi)$ is a {\em stably marked tree}
(\cite{GHvdP88}, Definition 1.2). This we mean that $\Delta$ is an
undirected tree and for each vertex $v\in V(\Delta)$ we have
\[
   {\rm val}(v):=|\psi^{-1}(v)|+|\{v'\in V(\Delta) \mid 
                     \{v,v'\}\in E(\Delta)\}| \geq 3.
\] 

Let us call an $L$-linear isomorphism $\lambda:X_L\liso\P^1_L$ a {\em
  chart}. Since $X_L=\P^1_L$ by definition, a chart may be represented by an
element in $\PGL_2(L)$. We call two charts $\lambda_1,\lambda_2$ {\em
  equivalent} if the automorphism
$\lambda_2\circ\lambda_1^{-1}:\P^1_L\liso\P^1_L$ extends to an automorphism of
$\P^1_{\OO_L}$, i.e.\ corresponds to an element of $\PGL_2(\OO_L)$. In
other words, an equivalence class of charts corresponds to a right coset in
$\PGL_2(\OO_L)\backslash\PGL_2(L)$.

Let $T$ denote the set of triples $t=(\alpha,\beta,\gamma)$ of pairwise distinct
elements of $D_L$. For $t=(\alpha,\beta,\gamma)$ 
we let $\lambda_t$ denote the unique chart such that
\[
   \lambda_t(\alpha)=0,\quad \lambda_t(\beta)=1,\quad
           \lambda_t(\gamma)=\infty.
\] 
Explicitly, we have
\begin{equation} \label{lambdateq}
   \lambda_t(x)=\frac{\beta-\gamma}{\beta-\alpha}\cdot\frac{x-\alpha}{x-\gamma},
\end{equation}
where we interpret this formula in the obvious way  if
$\infty\in\{\alpha,\beta,\gamma\}$. 
The equivalence relation $\sim$ on charts defined above induces an equivalence
relation on $T$, which we denote by $\sim$ as well. 

\begin{proposition} \label{treeprop1} 
  Let $(\cX,\cD)$ be the stable model of $(X_L,D_L)$. 
  \begin{enumerate}
  \item For all $t\in T$ the chart $\lambda_t$ extends to a proper
    $\OO_L$-morphism $\lambda_t:\cX\to\P^1_{\OO_L}$. Its reduction to the
    special fiber is a contraction morphism
    \[
         \lambdab_t:\Xb\to\P^1_{\FF_L}
    \]
    which contracts all but one component of $\Xb$ to a closed point.
  \item
    For every component $\Xb_v$ there exists $t\in T$ such that $\lambdab_t$
    does not contract $\Xb_v$ (and hence induces an isomorphism
    $\Xb_v\liso\P^1_{\FF_L}$). 
  \item
    The equivalence class of the chart $\lambda_t$ in (2) is uniquely
    determined by the component $\Xb_v$. We therefore obtain a bijection
    $V(\Delta)\cong T/_\sim$. 
  \end{enumerate}
\end{proposition}

\begin{proof} By combining Lemma 5 with the corollary to 
Lemma 4 of \cite{GHvdP88}, we
see that for every $t=(\alpha,\beta,\gamma)$ there exists a unique
proper $\OO_L$-morphism $\lambda:\cX\to\P^1_{\OO_L}$ such that
$\lambda(\alpha)=0$, $\lambda(\beta)=1$,
$\lambda(\gamma)=\infty$. Clearly, the restriction of $\lambda$ to the
generic fiber is equal to the chart $\lambda_t$. From now on we write
$\lambda=\lambda_t$.

The restriction of $\lambda_t$ to the special fiber is a proper
$\FF_L$-morphism $\lambdab_t:\Xb\to\P^1_{\FF_L}$. Since $(\Xb,\Db)$ is
stably marked, the morphism $\lambdab_t$ is uniquely determined by its
restriction to $\Db$. For $\delta\in D_L$ we have
$\lambdab_t(\bar{\delta})=\bar{\lambda_t(\delta)}$ by
construction. Therefore, $\lambdab_t$ is equal to the {\em generalized
  cross-ratio map} defined in \cite{GHvdP88}, \S~1. Statements (1)-(3)
follow immediately from the properties of this map proved in {\em
  loc.cit.}
\end{proof}

\begin{remark} \label{treerem1}
For $t=(\alpha,\beta,\gamma)\in T$ consider the map
\[
    \phi_t:D_L\to \P^1_{\FF_L}, \quad
     \delta \mapsto \lambdab_t(\bar{\delta}),
\]
where $\lambdab_t:\Xb\to\P^1_{\FF_L}$ is the map defined by Proposition
\ref{treeprop1}.(1). By the proof of the proposition we have
\[
    \phi_t(\delta)=\overline{\lambda_t(\delta)},
\]
where $\overline{\;\cdot\;}$ stands for the reduction map
$\P^1_L\to\P^1_{\FF_L}$. Together with formula \eqref{lambdateq}, this
shows that the collection of maps $(\phi_t)$ (which constitute a
finite amount of data) can be computed explicitly. By \cite{GHvdP88},
Proposition 1, the stably marked curve $(\Xb,\Db)$ can be
reconstructed effectively from the data $(\phi_t)_{t\in T}$. In
particular the following facts are shown in {\em loc.cit.}.
\begin{enumerate}
\item
  We have $t\sim t'$ if and only if $\phi_t=\phi_{t'}$. The maps
  $\phi_t$ determine the set $V(\Delta)$, via the bijection of Proposition
  \ref{treeprop1}.(3). 
\item
  For every $\delta\in D_L$ there exists a $t\in T$, unique up to
  $\sim$, such that $|\phi_t^{-1}(\phi_t(\delta))|=1$.  Moreover,
  $\bar{\delta}\in\Xb_v$, where $v\in V(\Delta)$ corresponds to $t$
  via the correspondence in (1). It follows that we can recover the map
  $\psi:D_L\to V(\Delta)$ from the maps $\phi_t$.
\item
  Fix $t\in T$ and let $v\in V(\Delta)$ correspond to $t$ via (1). Then the
  isomorphism $\Xb_v\liso\P^1_{\FF_L}$ induced by $\lambdab_t$ sends
  $\bar{\delta}$ to $\phi_t(\delta)$, for all $\delta\in D_L$. In this way, we
  can recover the divisor $\Db\subset\Xb$.
\end{enumerate}
\end{remark}

\begin{notation} \label{treenot} For every vertex $v\in V(\Delta)$ we choose
  $t\in T$ corresponding to $v$ via the bijection of Proposition
  \ref{treeprop1}.(3). Let $x_v:=\lambda_t^*(x)\in L(x)$ be the
  pullback of the standard coordinate $x$ of $X_L=\P^1_L$ via the
  chart $\lambda_t$. Equation \eqref{lambdateq} expresses $x_v$ in
  terms of the original coordinate $x$ and the triple $t=(\alpha,
  \beta, \gamma)$.

  Since $\cX$ is an integral, normal scheme and $\Xb_v\subset\cX$ is
  an irreducible closed subset of codimension one, the local ring of
  $\cX$ at the generic point of $\Xb_v$ is a discrete valuation
  ring. We denote the corresponding discrete valuation on $L(x)$ by
  $\eta_v$, where we normalize $\eta_v$ such that $\eta_v|_L$ is the
  standard valuation on $L$. Then $\eta_v$ is simply the Gauss
  valuation of $L(x_v)$ with respect to the parameter $x_v$. The
  residue field  of $\eta_v$ is naturally identified with
  the function field of $\Xb_v$. We have that
  \[ 
      F(\Xb_v)=\FF_L(\xb_v),
  \]
where $\xb_v$ denotes the image of $x_v$ in the function field $F(\Xb_v)$.
  In fact, $\xb_v$ is the pullback of the standard parameter of
  $\P^1_{\FF_L}$ via the isomorphism $\Xb_v\liso\P^1_{\FF_L}$ induced by
  $\lambdab_t$.  
\end{notation}

\subsection{} \label{kummer5}

As in \S~\ref{kummer2} we denote by $(\cX,\cD)$ the stable model of the
marked curve $(X_L,D_L)$, where $L$ is a finite extension of the splitting
field $L_0$ of $f$.  Let $\cY$ be the normalization of $\cX$ in the
function field of $Y_L$.  Corollary \ref{admisscor1} states that $\cY$
is a semistable model of $Y$ if $L$ is a sufficiently large tame
extension of $L_0$. The following proposition quantifies the degree of
$L/L_0$ and describes the special fiber $\Yb$ of $\cY$. 
 
Choose a prime element $\pi$ of $\OO_L$.  Consider $v\in
V(\Delta)$ and let  $x_v$ be the corresponding coordinate as in Notation
\ref{treenot}. Define
\[
   N_v:=\eta_v(f)/\eta_v(\pi), \qquad f_v:=\pi^{-N_v}f.
\]
Then $\eta_v(f_v)=0$ and we may consider the image $\bar{f}_v$ of
$f_v$ in the residue field $\FF_L(\xb_v)$ of the valuation
$\eta_v$. Let $n_v$ denote the order of the image of $\bar{f}_v$ in
the group $\FF_L(\xb_v)^\times/(\FF_L(\xb_v)^\times)^n$.

\begin{proposition} \label{kummerprop2}
\begin{enumerate}
\item
  Assume that the field $L$ contains all $n$th roots of unity. 
  Then the model $\cY$ of $Y_L$ is semistable if and only if $n\mid N_v$
  for all $v\in V(\Delta)$. 
\item Assume that the condition in (1) holds, and fix $v\in
  V(\Delta)$.  Then there is a bijection between the set of
  irreducible components of $\Yb_v:=\Yb|_{\Xb_v}$ and the set of
  elements $\bar{g}\in \FF_L(\xb_v)^\times$ satisfying
  \[
        \bar{g}^{n/n_v} = \bar{f}_v.
  \]
 The restriction of $\bar{\phi}$ to the irreducible component
    corresponding to $\bar{g}$ is the Kummer cover
   with equation
  \[
        \bar{y}_v^{n_v}=\bar{g},
  \]  
where $y_v=\pi^{-N_v/n}y$.  
\end{enumerate}
\end{proposition}

\begin{proof} By Theorem \ref{admissthm1} and the proof of Corollary
\ref{admisscor1}, $\cY$ is semistable if and only if the valuation
$\eta_v$ is unramified in the extension of function fields
$F(Y_L)/F(X_L)$ for all $v\in V(\Delta)$. If this is the case, the
irreducible components of $\Yb_v$ are in bijection with the discrete
valuations on $F(Y_L)$ extending $\eta_v$. The irreducible component
corresponding to an extension $\xi_v$ of $\eta_v$ to $F(Y_L)$ is the
smooth projective curve whose function field is the residue field of
$\xi_v$.  This
reduces the proof of the proposition to standard facts on the behavior
of valuations in Kummer extensions. For convenience we give the main
argument.

Assume that $n\mid N_v$ for some $v$. Then the element
$y_v:=\pi^{-N_v/n}y\in F(Y_L)$ generates the extension $F(Y_L)/F(X_L)$
and is a root of the irreducible polynomial
$F_v:=T^n-f_v\in L(x_v)[T]$.  The
polynomial $F_v$ is integral with respect to $\eta_v$. Its
reduction is separable and is the product of $n/n_v$ irreducible
factors of degree $n_v$, as follows:
\[
     \bar{F}_v = \prod_{\bar{g}^{n/n_v}=\bar{f}_v} (T^{n_v}-\bar{g}).
\]
(Here the hypothesis $\zeta_n\in L$ is used.)  It follows that $\eta_v$ is
unramified in the extension $F(Y_L)/F(X_L)$. Furthermore, the extensions of
$\eta_v$ are in bijection with the irreducible factors of $\bar{F}_v$. For
each extension the residue field extension is generated by the image of $y_v$,
which is a root of the corresponding irreducible factor of $\bar{F}_v$. This
proves (2) and the backward implication in (1). The forward implication in (1)
is left to the reader.
\end{proof}

\begin{corollary}\label{kummercor}
  Assume that $L$ contains the $n$th roots of unity and that the
  ramification index of $L/L_0$ is divisible by $n$. Then $Y_L$ has
  semistable reduction. The irreducible components of the reduction
  $\Yb$ are absolutely irreducible.
\end{corollary}

\section{Computing the inertial reduction} \label{mono}

We continue using the notation of the previous section. In particular,
$\phi:Y\to X=\P^1_K$ is a Kummer cover given by the equation
\[
        y^n=f(x)
\]
satisfying Assumption \ref{fnass}, $L_0/K$ is the splitting field of
$f$ and $L/L_0$ is a sufficiently large finite extension. The precise
meaning of `sufficiently large' is given by the condition of
Proposition \ref{kummerprop2}.(1). In this section we assume that the
possibly stronger condition from Corollary \ref{kummercor} holds.  

Let $(\cX,\cD)$ be the stable model of the marked curve $(X_L,D_L)$
and $\cY$ the normalization of $\cX$ in the function field of
$Y_L$. By Proposition \ref{kummerprop2} and Remark
\ref{admissrem2}.(1), $\cY$ is a quasi-stable model of $Y_L$. After
enlarging $L$ we may also assume that $L/K$ is a Galois extension. The
following assumption summarizes the requirements on $L$.

\begin{assumption}\label{Lass}
 We consider a finite
extension $L/K$ satisfying 
\begin{itemize}
\item  $L$ contains the splitting field $L_0$ of $f$ over $K$,
\item  $L$ contains a primitive $n$th root of $1$ and an
  $n$th root of $p$,
\item the extension $L/K$ is Galois.
\end{itemize}
\end{assumption}

As before we let $\Gamma=\Gal(L/K)$ denote the Galois group of $L/K$
and $I\lhd\Gamma$ the inertia subgroup. The group $\Gamma$ has a
natural semilinear action on the special fiber
$\Yb:=\cY\otimes_{\OO_L}\FF_L$. Recall that the {\em inertial
  reduction} of $Y$ (with respect to the quasi-stable model $\cY$) is
defined as the quotient $\Zb:=\Yb/\Gamma$. In this section we give a
concrete recipe how to compute $\Zb$. Our assumption is that the
extension $L/K$ together with the Galois group $\Gamma=\Gal(L/K)$
and its action on a chosen prime element $\pi$ of $L$ are
known explicitly.

Our strategy to compute $\Zb$ may be summarized as follows. It is
clear that the cover $\phi:Y\to X$ extends to a finite
$\Gamma$-equivariant morphism $\cY\to\cX$. Its restriction to the
special fiber is a finite $\Gamma$-equivariant map $\phib:\Yb\to\Xb$
between semistable curves over $\FF_L$. It induces a finite map
$\Zb\to\Wb:=\Xb/\Gamma$ between semistable curves over $\FF_K$. We
also write $\Zb_{\FF_L}:=\Yb/I$ and $\Wb_{\FF_L}:=\Xb/I$ for the
quotients by the action of the inertia group. Diagram \eqref{xywzdiag}
shows the relevant maps.
\begin{figure}
\begin{equation} \label{xywzdiag}
\begin{gathered}
\xymatrix@R-=3mm{
  \Yb \ar[dr] \ar[ddd] \\
      & \Zb_{\FF_L}=\Yb/I \ar[ddd]\ar[dr] \\
              &&\Zb=\Yb/\Gamma\ar[ddd]\\
  \Xb \ar[dr] \ar[ddd] \\
      & \Wb_{\FF_L}=\Xb/I \ar[dr]\ar[ddd] \\
      && \Wb=\Xb/\Gamma \ar[ddd] \\
  \Spec\FF_L \ar[dr] \\
      &\Spec\FF_L \ar[dr] \\
      && \Spec \FF_K }
\end{gathered}
\end{equation}
\end{figure}
Our strategy is to first describe the action of $\Gamma$ on
$\Xb$ (\S\S~\ref{mono1} and
\ref{mono2}), and then the maps $\Zb_{\FF_L}\to\Wb_{\FF_L}$
(\S~\ref{mono3}) and $\Zb\to\Wb$ (\S~\ref{mono4}).

\subsection{} \label{mono1}

Recall that $(\Xb,\Db)$ is the special fiber of the stable model
$(\cX,\cD)$ of the marked curve $(X_L,D_L)$. In particular, $\Xb$ is a
semistable curve of genus zero. Let $\Delta$ denote the tree of
components associated with $\Xb$. In \S~\ref{kummer2} we gave a
description of the tree $\Delta$ in terms of the divisor $D_L\subset
X_L$. It is clear from this description that the action of $\Gamma$ on
$V(\Delta)$ is determined, in an explicit way, by the action of
$\Gamma$ on $D_L$. (We refer to \S~\ref{exa14} for an explicit
example.)  We may therefore consider the action of $\Gamma$ on
$\Delta$ as known.

For a vertex $v\in V(\Delta)$ we let $\Gamma_v\subset\Gamma$ be the
stabilizer of the component $\Xb_v$ of $\Xb$ corresponding to $v$. The
subgroup $\Gamma_v$ consists exactly of those elements of $\Gamma$
leaving invariant  the set $\psi^{-1}(v)$ consisting of the branch points
$\alpha\in D_L$ specializing to $\Xb$.

The curve $\Wb=\Xb/\Gamma$ is a semistable curve over $\FF_K$ with
component graph $\Delta/\Gamma$.  Then $\Wb_v:=\Xb_v/\Gamma_v$ is the
irreducible component of $\Wb$ corresponding to the $\Gamma$-orbit of
$v$.  In order to compute $\Wb=\Xb/\Gamma$, it therefore suffices to
compute $\Wb_v=\Xb_v/\Gamma_v$, for each $v$.

\subsection{}\label{mono2}

Let us fix a vertex $v\in V(\Delta)$. The goal of Lemma \ref{monolem2} below
is to describe the action of $\Gamma_v$ on the curve
$\Xb_v$. We retain Notation \ref{treenot} and write
\begin{equation} \label{monoeq2}
        x_v=A(x)=\frac{ax+b}{cx+d},\qquad\text{with
       $A:=\begin{pmatrix} a & b \\ c & d \end{pmatrix}\in \PGL_2(L)$}.
\end{equation}

\begin{lemma} \label{monolem2} For $\sigma\in\Gamma_v$ the matrix
  $B_\sigma:=\sigma(A)A^{-1}$ lies in $\PGL_2(\OO_L)$. Furthermore, if
  $\psi_\sigma\in\Aut(\FF_L(\xb_v))$ denotes the automorphism induced by the
  action of $\sigma$ on $\Xb_v$, then
  \[
     \psi_\sigma(\xb_v)=\bar{B}_\sigma(\xb_v).
  \]
  Here $\bar{B}_\sigma\in\PGL_2(\FF_L)$ denotes the reduction of $B_\sigma$.
\end{lemma}

\begin{proof} An element $\sigma\in\Gamma=\Gal(L/K)$ acts canonically on
  $L(x)$, the function field of $X_L=\P^1_L$, by fixing the generator
  $x$. Therefore,
\[
     \sigma(x_v)=\sigma(A(x))=\sigma(A)(x)=\sigma(A)(A^{-1}(x_v))=
       B_\sigma(x_v).
\]
If $\sigma\in\Gamma_v$ then $\sigma$ fixes the Gauss valuation corresponding
to $x_v$ and hence $B_\sigma\in\PGL_2(\OO_L)$. The equality
$\psi_\sigma(\xb_v)=\bar{B}_\sigma(\xb_v)$ is a direct consequence. 
\end{proof}

\begin{remark}\label{monorem}
  Clearly, the map $\Gamma_v\to\Aut(\FF_L(\xb_v))$,
  $\sigma\mapsto\psi_\sigma$, is a group homomorphism. However, the map
  $\Gamma_v\to\PGL_2(\FF_L)$, $\sigma\mapsto\bar{B}_\sigma$, is {\em not}
  a group homomorphism. A straightforward computation shows that it obeys the
  rule
  \[
       \bar{B}_{\sigma\tau}=\sigma(\bar{B}_\tau)\cdot\bar{B}_\sigma.
  \]
The reason is that the restriction of $\psi_\sigma$ to $\FF_L$ need
not be trivial if $\sigma\notin I_v$.
It follows that the map $\sigma\mapsto\bar{B}_\sigma$ defines an
element of the set of nonabelian cocycles
  \[
      Z^1(\Gamma,\PGL_2(\FF_L)^{\rm opp}),
  \]
as defined in \cite{SerreCG}, I, \S~ 5.1. Of course, the restriction of this
cocycle to the inertia group $I\subset\Gamma$ is a group homomorphism.
\end{remark}  

\begin{lemma} \label{monolem3}
  For a suitable choice of the chart $\lambda_v$ we have
  \[
       \psi_\sigma(\xb_v)=a_\sigma\xb_v+b_\sigma, 
  \]
  with $a_\sigma,b_\sigma\in\FF_L$, for all $\sigma\in\Gamma_v$. In
  other words, $\psi_\sigma$ is an affine linear transformation for
  all $\sigma\in\Gamma_v$. 
\end{lemma}

\begin{proof} To prove the lemma we need to show the existence of an
$\FF_L$-rational point $p_1\in \Xb_v$ which is fixed by all $\sigma\in
\Gamma_v$.  Let $p_0:=\bar{\infty}\in\Xb$ denote the specialization of
the point $\infty\in X_L=\P^1_L$.  It is clear that $p_0$ is an
$\FF_L$-rational point fixed by $\Gamma$. If $p_0\in\Xb_v$ then
$p_1:=p_0$ satisfies the requirements.

 Otherwise, we let $p_1\in\Xb_v$ be the unique singular point of $\Xb$
 such that $p_0$ is contained in the connected component of
 $\Xb-\{p_1\}$ not containing $\Xb_v-\{p_1\}$.  In other words, $p_1$
 is the unique singular point of $\Xb$ contained in $\Xb_v$ which is
 ``nearest'' to $p_0$.  Since $\psi_\sigma\in \Aut(\FF_L(\xb_v))$, it
 follows that $p_1\in\Xb_v$ is an $\FF_L$-rational point which is
 fixed by the action of $\Gamma_v$. We now choose the chart
 $\lambda_v$ such that $p_1$ is the point $\xb_v=\infty$. This shows
 the statement of the lemma.  \end{proof}

\subsection{}\label{mono3}

We now describe how to compute the quotient $\Zb_{\FF_L}=\Yb/I$ of
$\Yb$ by the action of the inertia group, together with the map
$\Zb_{\FF_L}\to\Wb_{\FF_L}=\Xb/I$. By what was explained in
\S~\ref{mono1}, it suffices to consider the subcurve $\Yb_v:=\Yb|_{\Xb_v}$.

We choose a chart for $\Xb_v$ as in Lemma \ref{monolem3}. Recall that this
means that $\sigma\in I_v$ acts on the coordinate $\xb_v$ as
$\psi_\sigma(\xb_v)=a_\sigma \xb_v+b_\sigma$ with $a_\sigma , b_\sigma\in
\FF_L.$ Abusing notation, we also write $\psi_\sigma(\xb_v, \yb_v)$ for
the automorphism on $\Yb_v$ induced by $\sigma$.

Recall that $\Yb_v$ is given by the Kummer equation 
\begin{equation}\label{redeq}
\yb_v^n=\fb_v(\xb_v),
\end{equation}
where $y_v=\pi^{-N_v/n} y$ (Proposition \ref{kummerprop2}.(2)). The curve $\Yb_v$
is in general reducible. We prefer to work
with the reducible equation (\ref{redeq}) rather than the equation for the
irreducible components. This means that we work with the function algebra
$\FF_L(\xb_v)[\yb_v]/(\yb_v^n-\fb_v)$ instead of with the function field of
one of the irreducible components.

We have assumed that $L$ contains a primitive $n$th root of unity (Assumption
\ref{Lass}). It follows that the groups $G$ and
$I_v$ commute inside $\Aut_{\FF_L}(\Yb_v)$.  The quotient cover
\[
\Zb_{v,\FF_L}=\Yb_v/I_v\to \Wb_{v,\FF_L}=\Xb_v/I_v
\]
is therefore still Galois with Galois group $G/(I_v\cap G)$. 
In Propositions \ref{invprop1} and \ref{invprop2} below  we
compute a Kummer equation for this cover.

Our next goal is to compute an equation for the quotient curve of
$\Yb_v$ by the finite group $I_v$ explicitly. Being an inertia group
$I_v=P_v\rtimes C_v$ is an extension of a cyclic group $C_v$ of order
prime to $p$ by its Sylow $p$-subgroup $P_v$.  The following
proposition describes the action of $P_v$ on $\Yb_v$.

\begin{proposition}\label{invprop1} Write $\bar{P}_v=\{\psi_\sigma\mid \sigma\in
 P_v\}$ for 
the image of $P_v$ in $\Aut_{\FF_L}(\Yb_v)$.
\begin{enumerate}
\item 
For every $\sigma\in P_v$ we have that
\[
\psi_\sigma(\xb_v, \yb_v)=(\xb_v+b_\sigma, \yb_v)
\]
for some $b_\sigma\in \FF_L.$
\item The group $\bar{P}_v$ is an  elementary abelian $p$-group.
\end{enumerate}
\end{proposition}

\begin{proof} Let $\sigma\in P_v$.  The definition $y_v=\pi^{-N_v/n} y$
implies that $\psi_\sigma(\yb_v)=\gamma_\sigma \yb_v$. Since $\sigma$
has $p$-power order, it follows that $\gamma_\sigma$ is
trivial.  We have chosen the chart $\lambda_v$ such that $\psi_\sigma$
acts on $\Xb_v$ as affine linear transformation (Lemma
\ref{monolem3}). Statement  (1) follows. Moreover,   we may identify
$\bar{P}_v$ with a subgroup of $\FF_L$. This implies (2).  \end{proof}

Proposition \ref{invprop1} allows us to compute the quotient cover
$\Yb_v/P_v\to \Xb_v/P_v$. The coordinates $\yb_v$ and
\[
\ub_v:=\prod_{\sigma\in \bar{P}_v} \psi_\sigma(\xb_v)=\prod_{\sigma\in
  \bar{P}_v}(\xb_v+b_\sigma)
\]
 are
$\bar{P}_v$-invariant and  generate the function ring of
$\Yb_v/P_v$.  The rational function $\fb_v(\xb_v)$ is an element of
$\FF_L(\ub_v)$, hence we may write $\fb_v(\xb_v)=\bar{g}_v(\ub_v)$.
The function $\bar{g}_v$ is easily determined explicitly by comparison of
coefficients.  We conclude that the curve $\Yb_v/P_v$ is
given by the Kummer equation 
\[
 \yb_v^n=\bar{g}_v(\ub_v).
\]
The Kummer cover $\Yb_v/P_v\to \Xb_v/P_v$ is given by
$(\ub_v, \yb_v)\mapsto \ub_v$. Note that the degree of this cover is
still $n$, since the intersection $G\cap \bar{P}_v\subset
\Aut_{\FF_L}(\Yb_v)$ is trivial.  

\bigskip\noindent It remains to consider the quotient of
$\Yb_v/P_v$ by $I_v/P_v=C_v$, which is cyclic of order prime
to $p$.  We choose an element $\sigma\in I_v$ whose image generates
$C_v$, this defines a section $C_v\to I_v$.  Define $\mu$ as the order
of $\psi_\sigma$ considered as automorphism of $\Yb_v$ and $m$
as the order of $\psi_\sigma\in \Aut(\Xb_v)$. Then $m\mid
\mu$. Moreover, $(\mu/m)\mid n$ since $\psi_\sigma^{m}\in
\Aut(\Yb_v)$ is an element of $G$, which is cyclic of order
$n$. In particular, we have that
\[
|G\cap \langle\psi_\sigma\rangle|=\frac{\mu}{m}.
\]
The cover $\Yb_v/I_v\to
\Xb_v/I_v$ is a Galois cover with Galois group
$G/(G\cap I_v)$, which is cyclic of order 
$\bar{n}:=n/(\mu/m)=nm/\mu$.

If $m=1$ we have that $\psi_\sigma\in G$ and the cover
$\Zb_{v,\FF_L}\to \Wb_{v,\FF_L}=\Xb_v/I_v$ is given by
\[
\zb_v^{n/\mu}=\bar{g}_v(\ub_v), \qquad \text{ where } \zb_v=\yb_v^\mu. 
\]

We consider the case $m\neq 1$. Recall from Lemma \ref{monolem3} that
$\psi_\sigma\in \Aut_{\FF_L}(\Xb_v)$ is an affine linear
transformation of order $m$ with at least one $\FF_L$-rational fixed
point (which we assumed to be $\xb_v=\infty$).  It follows that
the second fixed point is also $\FF_L$-rational.  After a further
normalization of the chart, we may assume that it is $\xb_v=0$.  With
this choice of chart we have that
\[
\psi_\sigma(\xb_v, \yb_v)=(c \xb_v,\gamma \yb_v)
\]
for some $c, \gamma\in \FF_L$.  The definitions of $\mu$ and $m$ imply
that $m=\ord(c)$ and $\mu=\lcm(m,\ord(\gamma))$. It follows that
$\gamma^{\mu/m}=c^s\in \FF_L$ for some integer $s$.

Since $P_v$ is a normal subgroup of $I_v$, the automorphism
$\psi_\sigma$ descends to an automorphism of $\Xb_v/P_v$, which we
still denote by $\psi_\sigma$.  The definition of the coordinate
$\ub_v$ of $\Xb_v/P_v$ implies that the fixed points $\xb_v=0, \infty$
map to $\ub_v=0, \infty$, respectively. It follows that
$\psi_\sigma(\ub_v)=\tilde{c}\ub_v$. Since the order of $\sigma$ is
prime to $p$ and hence prime to $|P_v|$, we have that
$\ord(\tilde{c})=\ord(c)=m$. 
 We conclude
that the functions
\[
\zb_v:=\yb_v^{\mu/m}\ub_v^{-s}, \qquad \wb_v:=\ub_v^m
\]
are invariant under $I_v$. We find the following Kummer equation: 
\begin{equation}\label{quotientIveq}
\Zb_{v, \FF_L}:\qquad \zb^{\bar{n}}=\frac{\yb_v^n}{\ub_v^{s\bar{n}}}=\frac{\fb_v(\xb_v)}{\xb_v^{s\bar{n}}}.
\end{equation}
Since the function algebra of the quotient curve
$\Zb_{v,\FF_L}$ is generated by $\zb_v$ and $\wb_v$,
it follows that the right-hand side of (\ref{quotientIveq}) is a
rational function $\hb_v(\wb_v)\in \FF_L(\wb_v).$ As in the previous
step, it is easy to calculate $\hb_v$.

The following proposition summarizes the above discussion.

\begin{proposition}\label{invprop2}
\begin{enumerate}
\item  We may choose the chart $\lambda_v$ such that 
\[
\psi_\sigma(\xb_v, \yb_v)=(c \xb_v, \gamma \yb_v), 
\]
for suitable constants $c, \gamma\in \FF_L^\times$.
\item
The cover $\Zb_{v,\FF_L}\to \Wb_{v,\FF_L}$ is given by a Kummer equation
\[
\zb_v^{\bar{n}}=\hb_v(\wb_v),
\]
where 
\[
   \ub_v :=\prod_{\sigma\in\bar{P}_v}\psi_\sigma(\xb_v), \quad 
   \wb_v:= \ub_v^m, \quad 
   \zb_v:=\yb_v^{\mu/m}\ub_v^{-s}.
\]
Moreover, we have $m=\ord(c)$, $\mu=\lcm(m,
\ord(\gamma))$, $c^s=\gamma^{\mu/m}$ and $\bar{n}=n/(\mu/m)$.
\end{enumerate}
\end{proposition}

In \S~\ref{exa1} we give an example where the degree $\bar{n}$ of the
quotient Kummer cover is strictly smaller than $n$ (Remark \ref{nbarrem}).

\begin{remark}
  In the case that $\mu/m=n$ the Galois group $G$ of the cover $\Yb_v\to
  \Xb_v$ is contained in $\langle \psi_\sigma\rangle\subset I_v$. In this case
  the quotient curve $\Zb_{v, \FF_L}=\Yb_v/I_v$ is a union of curves of genus
  $0$, since each component is isomorphic to a quotient of $\Xb_v$.  It
  follows that $v$ does not contribute to the $L$-function, and we may
  disregard $v$ in the rest of the calculation. An example can be found in
  \S~\ref{exa14}.
\end{remark}

\subsection{}\label{mono4} 
In this section we describe how to compute the quotient curve
$\Zb=\Yb/\Gamma=\Zb_{\FF_L}/(\Gamma/I)$, together with the map $\Zb\to
\Wb=\Xb/\Gamma$. We write $\Gammab:=\Gamma/I\simeq \Gal(\FF_L/\FF_K)$. 

In \S~\ref{mono2} we have already described the action of $\Gamma$ on $\Xb$,
and therefore on the set of irreducible components. This action is induced by
the action of $\Gamma$ on the roots of the polynomial $f$, which is assumed to
be known. As a result, the action of $\Gammab=\Gamma/I$ on the irreducible
components of $\Wb_{\FF_L}=\Xb/I$ may therefore be considered as known. 

Let us choose $v\in V(\Delta)$. As before, we denote by
$\Wb_{v,\FF_L}$ (resp.\ $\Wb_v$) the irreducible component of
$\Wb_{\FF_L}$ (resp.\ of $\Wb$) corresponding to the $I$-orbit
(resp.\ to the $\Gamma$-orbit) of $v$.  Similarly, we write
$\Zb_{v,\FF_L}=\Zb_{\FF_L}|_{\Wb_{v,\FF_L}}$ and
$\Zb_v:=\Zb|_{\Wb_v}$.  Let $\Gammab_v\subset \Gammab$ be the
stabilizer of $\Wb_v$ and put $\FF_v=\FF_L^{\Gammab_v}$.

Recall from Proposition \ref{invprop2} that the cover
$\Zb_{v,\FF_L}\to\Wb_{v,\FF_L}$ is given birationally by a Kummer equation
\[
     \zb_v^{\bar{n}} = \hb_v,
\]
where $\hb_v\in \FF_L(\wb_v)$ is a rational function in the coordinate
$\wb_v$ for the projective line $\Wb_{v,\FF_L}$.

\begin{proposition} \label{Gammabprop}
\begin{enumerate}
\item
  The curve $\Wb_v$ is isomorphic to the projective line over $\FF_v$, and a
  coordinate $\wb_v'$ corresponding to such an isomorphism can be explicitly
  computed. 
\item
  The cover $\Zb_v\to\Wb_v$ is birationally given by a Kummer equation
  \[
       (\zb_v')^{\bar{n}}=\hb_v',
  \]
  where $\hb_v'$ is a polynomial in $\wb_v'$ with $\FF_v$-coefficients. The
  polynomial $\hb_v'$ can be explicitly computed. 
\end{enumerate}
\end{proposition}

\begin{proof}
Since $\Wb_v$ is a curve of genus zero over  $\FF_v$, the
first part of (1) follows from the fact that the Brauer group of the finite
field $\FF_v$ is trivial. However, in order to justify the second claim in (1)
it is better to give a more direct proof which does not use the Brauer group
(and therefore does not depend on $\FF_v$ being finite). 

By Proposition \ref{invprop2}, the function field of $\Wb_{v,\FF_L}$
is $\FF_L(\wb_v)$, where $\wb_v$ is an explicit polynomial in 
the chosen coordinate $\xb_v$ on $\Xb_v$. The semilinear action of $\Gammab_v$
is therefore given by a cocycle
\[
   (A_\tau)_\tau \in Z^1(\Gammab_v,{\rm PGL}_2(\FF_L)^{\rm opp}), 
\]
which can be explicitly computed from the knowledge of the cocycle from Remark
\ref{monorem}. Moreover, since $\wb_v$ is a polynomial in $\xb_v$, Lemma
\ref{monolem3} shows that $A_\tau$ corresponds to an affine linear
transformation, i.e.
\[
   A_\tau = \begin{pmatrix} \bar{a}_\tau&\bar{b}_\tau\\0&1 \end{pmatrix},
\]
with $\bar{a}_\tau,\bar{b}_\tau\in\FF_L$. To prove (1) we need to find a
coordinate $\wb_v'$ which is $\Gammab_v$-invariant. In other words, we need to
find a matrix 
\[
    B=\begin{pmatrix} \alpha&\beta\\ 0&1 \end{pmatrix} \in{\rm GL}_2(\FF_L)
\]
such that $A_\tau=\tau(B)B^{-1}$ for all $\tau\in\Gammab_v$. This translates
to 
\[
  \frac{\alpha}{\tau(\alpha)}=\bar{a}_\tau, \qquad
  \beta-\tau(\beta)=\bar{b}_\tau \tau(\alpha).
\]
In fact, it suffices to solve this equation for a generator $\tau$ of
$\Gammab_v$.  Clearly, solutions $\alpha,\beta\in\FF_L$ may be found
explicitly as in the proof of the additive and multiplicative versions of
Hilbert's Theorem $90$. This completes the proof of (1).

It remains to prove (2). By (1) we can write $\hb_v$ as a rational function in
$\wb_v'$ with coefficients in $\FF_L$.
There exists a  rational function $\hb_v''\in \FF_L(\wb_v')$ such that
\[
         \hb_v'=\hb_v(\hb_v'')^{\bar{n}},
\]
is a polynomial in $\FF_L[\wb_v']$ which does not have any nontrivial
factors which are $\bar{n}$th powers. We set
$\zb_v':=(\hb_v'')\zb_v$. The cover $\Zb_{v,\FF_L}\to\Wb_{v,\FF_L}$ is
now given by the Kummer equation
\begin{equation} \label{Gammabpropeq1}
        (\zb_v')^{\bar{n}} = \hb_v'.
\end{equation}

For $\tau\in \Gammab_v$, we write $\psi_\tau$ for the (semilinear)
automorphism of $\Zb_{v,\FF_L}$ induced by $\tau$.
 We claim that for
any element $\tau\in\Gammab_v$ we have
\begin{equation} \label{Gammabpropeq2}
       \psi_\tau(\zb_v')=\bar{q}_\tau\cdot\zb_v', \qquad
     \text{with $\bar{q}_\tau\in\FF_L[\wb_v']$.}
\end{equation}

To see this, note that the extension 
\begin{equation}\label{Gammabpropeq3}
\FF_L(\Zb_{v, \FF_L})\supset \FF_v(\Wb_v)\simeq \FF_v(\wb_v')
\end{equation}
 of functions rings is a Galois extension. Recall that the Galois group
 $\bar{G}:=\Gal(\Zb_{v,\FF_L}/\Wb_{v,\FF_L})$ is cyclic of order
 $\bar{n}$. Since $\FF_L$ contains the $\bar{n}$th roots of unity,
 $\bar{G}$ is a normal subgroup of the Galois group of the extension
 (\ref{Gammabpropeq3}), which is a quotient of $\Gamma_v$. It follows
 that $\psi_{\tau}(\zb_v')$ is a Kummer generator of $\Zb_{v,
   \FF_L}/\Wb_{v, \FF_L}$.  Kummer theory implies that
\begin{equation}\label{Gammabpropeq4}
   \psi_\tau(\zb_v')=\bar{q}_\tau\cdot(\zb_v')^{m_\tau},
\end{equation}
where $m_\tau\in\{1,\ldots,\bar{n}-1\}$ represents the character
$\chi:\Gammab_v\to(\ZZ/\bar{n}\ZZ)^\times$ which determines the action
of $\Gammab_v$ on $\bar{G}$ by conjugation.  The claim
(\ref{Gammabpropeq2}) states that the character $\chi$ is trivial.

To prove that $\chi$ is trivial, we consider the action of $\psi_\tau$
on the polynomial $\hb_v'$.  Recall that $\hb_v'$ is a polynomial which does
not have any nontrivial factors that are $\bar{n}$th powers. It
follows that the roots of $\hb_v'$ are branched in the Kummer cover
$\Zb_{v,\FF_L}\to \Wb_{v,\FF_L}$.  More precisely, the roots of
$\hb_v'$ are  the images of the branch points of the cover
$Y\to X$ that specialize to $\Xb_v$. In particular, it follows that
$\Gamma_v$ acts on the set of roots of $\hb_v'$.  

It also follows that the order of vanishing of a zero of $\hb_v'$ is
equivalent $\pmod{\bar{n}}$ to the order of vanishing of the
corresponding zero of the polynomial $f$ describing the Kummer cover
$Y\to X$.  Since $Y\to X$ is defined over $K$ it follows that any two
roots of $\hb_v'$ which are conjugate under the action of $\Gamma_v$
have the same order of vanishing in $\hb_v'$. The coordinate $\wb_v'$
is already invariant under $\tau$. We conclude that
\[
\psi_\tau(\hb_v')=\bar{q}_\tau\cdot\hb_v'\qquad \text{ with
}\bar{q}_\tau\in \FF_L^\times,
\]
for all $\tau\in \Gammab_v$. With \eqref{Gammabpropeq1} it follows
that $m_\tau$ in (\ref{Gammabpropeq4}) is trivial for all $\tau\in
\Gammab_v$, and hence that the character $\chi$ is trivial. This
proves the claim \eqref{Gammabpropeq2}.

Replacing $\zb_v'$ with $\gamma\zb_v'$, for some
$\gamma\in\FF_L^\times$, has the effect of replacing $\bar{q}_\tau$
with $\bar{q}_\tau\tau(\gamma)\gamma^{-1}$. Using again Hilbert's
Theorem 90, we may assume that $\bar{q}_\tau=1$, i.e.\ that $\zb_v'$
is invariant under the action of $\Gammab_v$.

The extension of function rings $F(\Zb_v)/F(\Wb_v)=\FF_v(\wb_v')$ has degree
$\bar{n}$, which is the same as the degree of the Kummer equation for
$\zb_v'$. We conclude that $\zb_v'$ is a generator of the extension of
function rings $F(\Zb_v)/F(\Wb_v)=\FF_v(\wb_v')$. The proof of the
proposition is now complete. 
\end{proof}

Proposition \ref{Gammabprop} gives an explicit description of the
(possibly reducible) curves $\Zb_v=\Zb|_{\Wb_v}$.  Remark
\ref{admissrem2}.(2) implies that $\Zb_v$ is smooth. It follows that
the normalization $\pi:\Zb^{(0)}\to \Zb$ is the disjoint union of the
curves $\Zb_v$, where $v$ runs over a subset of $V(\Delta)$
representing the $\Gamma$-orbits.  We therefore have
an explicit description of the normalization $\Zb^{(0)}$ as well.

 As explained in \S~\ref{cohomology}, it remains to describe the
 singular locus $\Zb^{(1)}:=\pi^{-1}(\Zb^{\rm sing})\subset\Zb^{(0)}$.
 Remark \ref{admissrem2} implies that $\Zb^{(1)}$ is the inverse image
 of $\Wb^{(1)}\subset\Wb$ under the map $\Zb\to\Wb$, where $\Wb^{(i)}$
 is defined analogously to $\Zb^{(i)}$ for $i=0,1$.  Since the map
 $\Zb^{(0)}\to\Wb^{(0)}$ has an explicit description as a disjoint
 union of Kummer covers, it suffices to describe the closed subset
 $W^{(1)}\subset\Wb^{(0)}$. Since $\Wb=\Xb/\Gamma$, an explicit
 description of $\Wb^{(1)}\subset\Wb^{(0)}$ can immediately be derived
 from the inclusion $\Xb^{(1)}\subset\Xb^{(0)}$. This is easy using
 the description of $\Xb$ as a tree of projective lines in
 \S~\ref{kummer2}.



\section{Example I}\label{exa1}

In this section and the next we compute the local $L$-factor and the
exponent of conductor of two  superelliptic curves.

\subsection{}

We consider the Kummer cover $\phi:Y\to X=\P^1_K$ over $K:=\Q_3$ given by the
equation
\[
    y^4 = f(x) = (x^2-3)(x^2+3)(x^2-6x-3).
\]
The branch points of $\phi$ are the six roots of $f$ (with
ramification index $4$) and the point at $\infty$ (with ramification
index $2$). The Riemann--Hurwitz formula shows that the genus of $Y$
is $7$.

The splitting field of $f$ over $K$ is the biquadratic extension
$L_0:=K(i,3^{1/2})$, where $i$ is a fourth root of unity and $3^{1/2}$ is a
square root of three. In fact, the roots of $f$ are
\[
    \pm 3^{1/2}, \pm i 3^{1/2},\alpha,\alpha',
\]
where $\alpha=3-2\cdot 3^{1/2}, \alpha'=3+2\cdot 3^{1/2}\in L_0$ are the two
roots of $x^2-6x-3$. Note that $K(i)/K$ is the maximal unramified subextension
and that the residue field of $K(i)$ (and of $L_0$) is the field $\FF_9$ with
$9$ elements. 

Let $L:=L_0(3^{1/4})$ be the extension obtained by adjoining a square
root $3^{1/4}$ of $3^{1/2}$. Since $K(i)$ already contains all $4$th
roots of unity, we see that $L/K$ is a Galois extension whose Galois
group $\Gamma$ is the dihedral group of order $8$. The inertia
subgroup $I\lhd\Gamma$ is the unique cyclic subgroup of order
$4$. Moreover, $L$ satisfies Assumption \ref{Lass} and $Y_L$ has
semistable reduction over~$L$.

\subsection{}\label{exa12}

Let $(\cX,\cD)$ denote the stably marked model of $(X_L,D_L)$ and $(\Xb,\Db)$
the special fiber of $(\cX,\cD)$, see \S~\ref{admiss2}. We note that 
\[
     \frac{\alpha- 3^{1/2}}{3^{1/2}}\equiv 0\pmod{3^{1/4}},\quad
     \frac{\alpha'-(- 3^{1/2})}{3^{1/2}}\equiv 0\pmod{3^{1/4}},
\]
and that there are no further congruences between the elements of $D_L$.
Following Remark \ref{treerem1} one easily sees that $\cX$ is given by the
three charts $\lambda_i:X_L\to\P^1_L$, $i=1,2,3$ corresponding to the
parameters
\[
    x_1:=3^{-1/2}x,\quad x_2:=\frac{x-3^{1/2}}{3},\quad
    x_3:=\frac{x+3^{1/2}}{3}.
\]
Let $\Xb_i\subset \bar{X}$ be the irreducible component
corresponding to $\lambda_i$. Then $\bar{X}$ looks as follows:
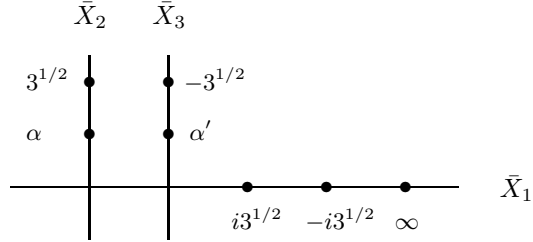
\[
\setlength{\unitlength}{0.7mm}
\begin{picture}(80,50)

\put(0,10){\line(1,0){85}}
\put(15,0){\line(0,1){35}}
\put(30,0){\line(0,1){35}}

\put(15,20){\circle*{2.2}}
\put(15,30){\circle*{2.2}}
\put(30,20){\circle*{2.2}}
\put(30,30){\circle*{2.2}}
\put(45,10){\circle*{2.2}}
\put(60,10){\circle*{2.2}}
\put(75,10){\circle*{2.2}}

\put(93,8){\smaller[1]$\Xb_1$}
\put(12,41){\smaller[1]$\Xb_2$}
\put(27,41){\smaller[1]$\Xb_3$}

\put(3,28.5){\smaller[2]$3^{1/2}$}
\put(3,19){\smaller[2]$\alpha$}
\put(33,28.5){\smaller[2]$-3^{1/2}$}
\put(34,19){\smaller[2]$\alpha'$}
\put(42,2){\smaller[2]$i3^{1/2}$}
\put(56,2){\smaller[2]$-i3^{1/2}$}
\put(73,2){\smaller[2]$\infty$}

\end{picture}
\]
In this picture the dots indicate the position of the points
$\bar{\alpha}_i\in \bar{D}\subset \bar{X}$. Next to the dots one finds
the value of the corresponding point $\alpha_i\in D_L\subset
X_L=\P^1_L$.

\subsection{}\label{exa13}

Let $\cY$ denote the normalization of $\cX$ in the function field of $Y_L$. We
 use Proposition \ref{kummerprop2} to show that $\cY$ is a semistable
model of $Y_L$ and to describe its special fiber $\Yb$.

Let $\eta_i$ denote the discrete valuation corresponding to the
component $\Xb_i$ on the function field $F(X_L)=L(x)$, where we
normalize $\eta_i$ by $\eta_i(3)=1$. Set $N_i:=\eta_i(f)$. For $i=1$
we write
\[
   f(x)=f(3^{1/2}x_1)=3^3(x_1^2-1)(x_1^2+1)(x_1^2-2\cdot 3^{1/2}x_1-1),
\]
from which we conclude that
\[
    \eta_1(f)=3, \quad \bar{f}_1=(\xb_1^2-1)^2(\xb_1^2+1).
\]
Similarly, we check that for $i=2,3$ we have
\[
    \eta_i(f)=4, \quad \bar{f}_i=2\xb_i(\xb_i-1).
\]
By the first part of Proposition \ref{kummerprop2} it follows that
$\cY$ is semistable. The second part of the proposition implies that
there is a unique irreducible component $\bar{Y}_i$ of $\bar{Y}$ lying
above $\bar{X}_i$. The restriction $\Yb_i\to\Xb_i$ is the Kummer
cover with equation $\bar{y}_i^4=\bar{f}_i$, for $i=1,2,3$.  Note that
the genus of $\bar{Y}_1$ is equal to $3$, whereas
$\bar{Y}_2$ and $\bar{Y}_3$ have genus $1$.

To describe $\bar{Y}$ it remains to describe the singular locus of
$\bar{Y}$.  By Remark \ref{admissrem2}.(2), the singular locus of
$\bar{Y}$ is precisely the inverse image of the singular locus of
$\bar{X}$. The latter is contained in the component $\bar{X}_1$, and
consists of the two points with $\bar{x}_1=\pm 1$.  Note that the
points above $\bar{x}_1=\pm 1$ have ramification index $2$ in the
cover $\bar{Y}_1\to \bar{X}_1$. Hence $\Yb$ contains $2\cdot (4/2)=4$
singular points: two intersection points of $\Yb_2$ with $\Yb_1$ and
two intersection points of $\Yb_3$ with $\Yb_1$.
The curve $\bar{Y}$ therefore  looks as follows.
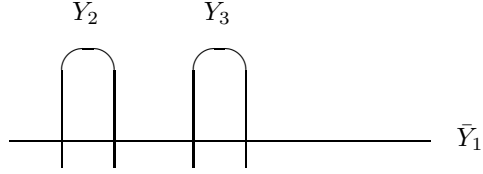
\[
\setlength{\unitlength}{0.7mm}
\begin{picture}(90,45)

\put(0,10){\line(1,0){80}}
\put(15,5){\oval(10,45)[t]}
\put(40,5){\oval(10,45)[t]}

\put(85,9){\smaller[1]$\bar{Y}_1$}
\put(12,33){\smaller[1]$\bar{Y}_2$}
\put(37,33){\smaller[1]$\bar{Y}_3$}

\end{picture}
\]
Note that the arithmetic genus of $\bar{Y}$ equals $3+1+1+2=7$, which is
equal to the genus of $Y$, as it should be. 

\subsection{}\label{exa14}

We now look at the action of $\Gamma=\Gal(L/K)$ on $\bar{Y}$. Let
$\sigma,\tau\in\Gamma$ be the two generators given by
\[\begin{split}
    \sigma(3^{1/4})=i\cdot 3^{1/4}, &\quad \sigma(i)=i,\\
    \tau(3^{1/4})=3^{1/4}, & \quad \tau(i)=-i.
\end{split}\] 
Recall that the inertia subgroup group $I\subset\Gamma$ is cyclic of
order $4$, hence $I$ is generated by $\sigma$.

Following the strategy of \S~\ref{mono} we first study the action of
$I=\langle\sigma\rangle$ on $(\bar{X},\bar{D})$, which is determined
by its action on the set $D_L$.

The element $\sigma\in
I$ acts  as an involution on $D_L$, as
follows:
\[
    3^{1/2} \leftrightarrow -3^{1/2}, \quad  
   i3^{1/2} \leftrightarrow -i3^{1/2}, \quad
   \alpha  \leftrightarrow \alpha'.
\]
It follows that the automorphism $\psi_\sigma$ of $\Xb$ maps the
component $\bar{X}_1$ of $\bar{X}$ to itself and interchanges the two
components $\bar{X}_2,\bar{X}_3$.  We conclude that 
$\psi_\sigma$ also  fixes the component $\bar{Y}_1$
of $\bar{Y}$ and interchanges
$\bar{Y}_2$ with $\bar{Y}_3$.

As a second step we determine the quotients $\Zb_{\FF_L}=\Yb/I\to
\Wb_{\FF_L}=\Xb/I$.
The definition of $x_1$ as $x_1=x/3^{1/2}$ implies that the
restriction of $\psi_\sigma$ to $\Xb_1$ is given by
$\psi_\sigma(\bar{x}_1)=-\bar{x}_1$.  The coordinate $\bar{y}_1$ is the
image in $\FF_L(\bar{Y}_1)$ of $y_1:=\pi^{-N_1/n}y=3^{-3/4}y$ (Proposition
\ref{kummerprop2}.(2)). It follows that
\[
   \psi_\sigma(\xb_1, \bar{y}_1)=(-\xb_1, i\yb_1).
\]
Therefore the Kummer equation for $\Yb_1/I_1\to \Xb_1$ from
Proposition \ref{invprop2}.(2) is given by
\begin{equation}\label{quoIeq1}
\zb_1^2=\wb_1(\wb_1+1), \quad \wb_1=\xb_1^2, \,
\zb_1=\yb_1^2\xb_1/(\xb_1^2-1).
\end{equation}
This implies that $\Zb_{1,\FF_9}\cong\P^1_{\FF_9}$ has genus zero. 

\begin{remark}\label{nbarrem} Note that  $\psi_\sigma$ considered as 
automorphism of $\Yb_1$ has order $4$, which is strictly larger than
the order of the corresponding automorphism of $\Xb_1$. This is the
reason why the quotient Kummer cover $\Zb_{1,\FF_L}\to \Wb_{1,\FF_L}$ has
degree $\bar{n}=4/2=2$.
\end{remark}

A similar analysis shows that $\psi_\sigma(\xb_2)=\xb_3$ and 
$\psi_{\sigma^2}(\xb_2)=\xb_2$.  The restriction of $\psi_{\sigma^2}$
to $\bar{Y}_2\cup \bar{Y}_3$ is the identity since $y_2=y/3$. We have
already seen that $\psi_\sigma$ interchanges $\Yb_2$ and $\Yb_3$. We
conclude that $\bar{Z}_{2,\FF_9}:=(\bar{Y}_2\cup \bar{Y}_3)/I$ is an
isomorphic copy of $\bar{Y}_2$ (or $\bar{Y}_3$). The quotient cover
$\Zb_{2, \FF_L}\to \Wb_{2,\FF_L}$ is the same as the original cover
$\Yb_2\to \Xb_2$, 
i.e.\
\begin{equation}\label{quoIeq2}
\zb_2^4=2\wb_2(\wb_2-1), \qquad \zb_2:=\yb_2=\yb_3, \, \wb_2=\xb_2+\xb_3.
\end{equation}

It follows that the quotient curve $\bar{Z}_{\FF_9}:=\bar{Y}/I$ is a
semistable curve over $\FF_9$ consisting of two irreducible components
$\bar{Z}_{1,\FF_9}$ and $\bar{Z}_{2,\FF_9}$ intersecting each other in
two points, as follows.
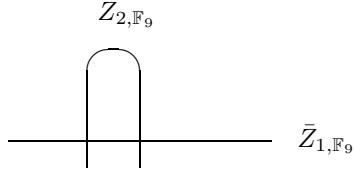
\[
\setlength{\unitlength}{0.7mm}
\begin{picture}(60,45)

\put(0,10){\line(1,0){50}}
\put(20,5){\oval(10,45)[t]}

\put(55,9){\smaller[1]$\bar{Z}_{1,\FF_9}$}
\put(17,33){\smaller[1]$\bar{Z}_{2,\FF_9}$}

\end{picture}
\]
The arithmetic genus of $\bar{Z}_{\FF_9}$ is equal to
$g(\bar{Z}_{\FF_9})=g(\bar{Z}_{1,\FF_9})+g(\bar{Z}_{2,\FF_9})+1=0+1+1=2$.

\subsection{}\label{exa15}

It remains to determine the semilinear action of
$\Gammab=\Gamma/I=\langle \bar{\tau}\rangle$ on $\Zb_{\FF_L}$ and the
quotient $\Zb:=\Zb_{\FF_L}/\Gammab=\Yb/\Gamma$.  By considering the
action of $\tau$ on the branch points of $\phi$ as in \S~\ref{exa14},
we see that $\psi_{\bar{\tau}}$ acts trivially on the graph $\Delta$
of components of $\Xb$. Since there is a unique irreducible component
of $\Yb$ above $\Xb$, $\psi_{\bar{\tau}}$ also acts trivially on the
graph of components of $\Yb$.

From the proof of Proposition \ref{Gammabprop} it follows that
$\bar{\tau}$ leaves the coordinates $\zb_i, \wb_i$ defined in
(\ref{quoIeq1}) and (\ref{quoIeq2}) invariant.  We conclude that
$\Zb_{\FF_L}$ is already the correct model over $\FF_3$. Note that the
$\langle\bar{\tau}\rangle\simeq \Gal(\FF_9/\FF_3)$ acts semilinearly
on $\Zb=\Zb_{\FF_L}$. For example, the singular locus of $\bar{Z}$
consists of two geometric points which are conjugate over the
quadratic extension $\F_9/\F_3$. This completes our description of
$\bar{Z}$.

\subsection{}

We can now write down the local $L$-factor of the curve $Y/\Q_3$. By
Corollary \ref{ssredcor1}, the local factor is
\[
       L_3(Y,s)=P_1(\bar{Z},3^{-s}), 
\]
where
\[
     P_1(\bar{Z},T):= \det\big(1-{\rm Frob}_3\cdot T|H^1(\bar{Z},\Q_\ell)\big)
\]
and where ${\rm Frob}_3:\bar{Z}_{\F_3}\to\ \bar{Z}_{\F_3}$ is the
$\F_3$-Frobenius endomorphism.

The normalization of $\bar{Z}$ is equal to the disjoint union of
$\bar{Z}_1\cong\P^1_k$ and $\bar{Z}_2$. Lemma \ref{cohomologylem}.(1)
implies that 
\[
H^1_\et(\Zb_k,\Q_\ell)=H^1(\Delta_{\Zb_k})\oplus H^1_\et(\Zb_{2,k},\Q_\ell).
\]
In \S~\ref{exa15} we have seen that that ${\rm Frob}_3$ fixes the two
irreducible components $\bar{Z}_1$ and $\bar{Z}_2$ of $\bar{Z}$ and
interchanges the two singular points. Lemma \ref{cohomologylem}.(2)
therefore implies that the corresponding factor of $P_1(\bar{Z},T)$ is equal to
\[
          1+T.
\]

The second factor is  the numerator of the  zeta function of the
genus-one curve $\bar{Z}_{2}$ given by (\ref{quoIeq2}).  
Since the number of $\F_3$-rational points is
\[
      |\bar{Z}_{2}(\F_3)|=4=1+3,
\]
it follows that
\[
       P_1(\bar{Z},T) = (1+T)(1+3T^2).
\]

\subsection{}

We use our description of the stable reduction of $Y$ to
compute the exponent of the conductor of the $\Gamma_K$-representation
$H^1(Y_{\bar{K}},\Q_\ell)$. Since $Y$ achieves semistable reduction
over a tame extension of $K=\QQ_3$ it follows from Corollary 
\ref{ssredcor2} and the above calculations that
\[
f_{Y/K}=2g(Y)-\dim H^1_{\et}(\bar{Z}_k,\Q_\ell) = 14 -3 =11.
\]

\section{Example II}\label{exa2} 

As a second example we consider the curve $Y$ over $K=\QQ_2$ given by 
\[
      y^3 = f(x):=x^4-x^2+1.
\]
We will see that in this case the extension $L/\QQ_2$ over which $Y$ acquires
stable reduction  is wildly ramified.  

\subsection{}

The ramification divisor $D\subset X:=\P^1_K$ has degree $5$ and
consists of the zero set of $f$ together with $\infty$, hence
$g(Y)=3$.  As $f$ is the $12$th cyclotomic polynomial, its zero set
is $\{\pm \zeta, \pm \zeta^5\}$, where $\zeta$ is a chosen primitive
$12$th root of unity. The splitting field of $f$ is
$L_0:=\QQ_2(\zeta)$.  We set $L:=L_0(2^{1/3})$, where $2^{1/3}$ is a
$3$rd root of $2$. Since $L_0$ contains the third root of unity
$\zeta_3:=\zeta^4$, the extension $L/K$ is Galois and its Galois group
$\Gamma:=\Gal(L/K)$ is the dihedral group of order $12$. Its inertia
subgroup is $I:=\Gal(L/K(\zeta_3))$, which is the cyclic subgroup of
$\Gamma$ of order $6$. In particular, $L/K$ is wildly ramified. The
residue field $\FF_L$ of $L$ is  $\FF_4$,
and is generated over $\FF_2$ by the image $\zetab$ of $\zeta_3$.
 Assumption \ref{Lass} is satisfied, therefore 
the curve $Y_L$ has semistable reduction over $L$.

As in \S~\ref{exa12}  we find that  the special fiber $\Xb$ 
of the stable model $(\cX,\cD)$ of
$(X_L,D_L)$ looks as follows:
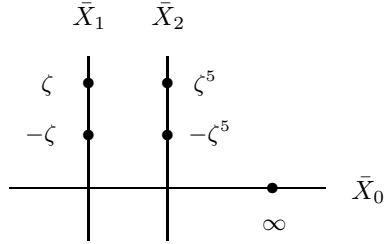
\[
\setlength{\unitlength}{0.7mm}
\begin{picture}(80,50)

\put(0,10){\line(1,0){60}}
\put(15,0){\line(0,1){35}}
\put(30,0){\line(0,1){35}}

\put(15,20){\circle*{2.2}}
\put(15,30){\circle*{2.2}}
\put(30,20){\circle*{2.2}}
\put(30,30){\circle*{2.2}}
\put(50,10){\circle*{2.2}}

\put(65,8){\smaller[1]$\Xb_0$}
\put(12,41){\smaller[1]$\Xb_1$}
\put(27,41){\smaller[1]$\Xb_2$}

\put(6,28.5){\smaller[2]$\zeta$}
\put(3,19){\smaller[2]$-\zeta$}
\put(35,28.5){\smaller[2]$\zeta^5$}
\put(34,19){\smaller[2]$-\zeta^5$}
\put(48,2){\smaller[2]$\infty$}

\end{picture}
\]

We may choose the parameters $x_i$ for the components $\Xb_i$ as follows
\begin{equation}\label{xidef}
x_0:=x,\qquad x_1:=\frac{x-\zeta}{2}, \qquad x_2:=\frac{x-\zeta^5}{2}.
\end{equation}
The choice of $x_1$ differs from the convention in Notation \ref{treenot} by a
unit. This  leads to slightly easier formulas afterwards.

Proposition \ref{kummerprop2} yields as Kummer equation for
$\Yb_i:=\Yb|_{\Xb_i}$:
\begin{align}
y_0&:=y,  &\yb_0^3&=\fb_0(\xb_0):=(\xb_0^2+\xb_0+1)^2,\label{Zb0eq}\\
y_1&:=2^{2/3}y, &\yb_1^3&=\fb_1(\xb_1):=\xb_1(\xb_1+\zetab),\label{Zb1eq}\\
y_2&:=2^{2/3}y,&\yb_2^3&=\fb_2(\xb_2):=\xb_2(\xb_2+\zetab^2)\label{Zb2eq}.
\end{align}
Note that $\Yb_i$ is irreducible and has genus $1$  for $i=1,2,3$

\subsection{}\label{exa23}

We now describe the action of $\Gamma=\Gal(L/K)$ on $\Xb$ and $\Yb$
and determine the quotient curve $\Zb=\Yb/\Gamma$.
 For convenience we
choose generators $\sigma, \tau$ of $\Gamma$ as follows
\begin{align}\label{Gammaeq}
\sigma(i)&=-i, &\sigma(2^{1/3})&=\zeta_32^{1/3},
&\sigma(\zeta_3)&=\zeta_3,\\
\tau(i)&=i, &\tau(2^{1/3})&=2^{1/3},
&\tau(\zeta_3)&=\zeta_3^2.
\end{align}
 Note that $\sigma$ generates $I$
and the image of $\tau$ generates $\Gammab:=\Gamma/I$.

Since $x_0=x$ and $y_0=y$ it follows that $\Gamma$ leaves these
coordinates invariant. We conclude that $\Wb_0:=\Xb_0/\Gamma$ is
isomorphic to the projective line over $\FF_2$ with parameter
$\xb_0$. Similarly, $\Zb_0:=\Yb_0/\Gamma$ is simply the $\FF_2$-model
of $\Yb_0$ given by the equation (\ref{Zb0eq}).

We describe the action of $\Gamma$ on the graph $\Delta$ of
irreducible components of $\Xb$. Since $\Gamma$ permutes the primitive
$12$th roots of unity, the components $\Xb_1$ and $\Xb_2$ are
interchanged. The choice of coordinates in (\ref{xidef}) implies that
$\psi_\tau(\xb_1)=\xb_2$, and conversely.  Since $\zeta^5=\zeta_3\cdot
\zeta$, the stabilizer $\Gamma_i$ of $\Xb_i$ is the inertia group $I$
for $i=1,2$. 

Obviously, $\Gamma$ permutes the components $\Yb_1$ and $\Yb_2$ as
well. We are reduced to computing the quotient $\Zb_1:=\Yb_1/I$. The
definition of the coordinates in (\ref{xidef}) and (\ref{Zb1eq})
implies that
\[
\psi_\sigma(\xb_1, \yb_1)=(\xb_1, \zetab \yb_1),
\]
since $(\zeta-\sigma(\zeta))/2=(\zeta-\zeta^7)/2=\zeta\equiv
\zeta_3\pmod{2}$.  Therefore $\psi_{\sigma^2}$ generates the Galois
group of $\Yb_1\to \Xb_1$ and $\Wb_1=\Yb_1/I$ is a projective line
over $\FF_4$ with coordinate $\wb_1:=\xb_1(\xb_1+\zetab).$ 

The corresponding component of $\Zb=\Yb/\Gamma$ is
$\Zb_3:=(\Zb_1\coprod\Zb_2)/\Gal(\FF_4/\FF_2)$. The $\Zb_3$ is
isomorphic to $\P^1_{\FF_4}$ considered as a curve over $\FF_2$ and
is not absolutely irreducible. Since $\Zb_1$ has
genus $0$, the curve $\Zb_3$ does not contribute to the \'etale
cohomology of $\Zb$. Since
there are no loops, the contraction map $\Zb\to\Zb_0$ induces an
isomorphism on $H^i_\et$.  

The curve  $\Zb_0=\Yb_0/\Gamma$ is the smooth curve of
genus $1$ over $\FF_2$ given by \eqref{Zb0eq} with
$|\Zb_0(\FF_2)|=3.$ 
We conclude that the zeta function of $\Zb$ is
\[
     Z(\Zb,T)=\frac{1+2T^2}{(1-T)(1-2T)}.
\]

\subsection{} \label{exawildconductor}

It remains to compute the exponent of conductor $f_{Y/K}$. Since the
extension $L/K$ is wildly ramified, Corollary \ref{ssredcor2} does not
apply and we have to use the formula of Theorem \ref{ssredthm3}. Recall that 
\[
       f_{Y/K}=\epsilon +\delta,
\]
where $\epsilon=2g_Y-\dim H^1_\et(\Zb_k,\QQ_\ell)$ and  $\delta$ is the
Swan conductor. The results from  \S~\ref{exa23} show that $\dim
H^1_\et(\Zb_k,\QQ_\ell)=2$ and therefore that $\epsilon=4$. 

Let $(\Gamma_i)_{i\geq 0}$ be the filtration of $\Gamma$ by higher
ramification groups. Then $\Gamma_0=I$ is the inertia group and $\Gamma_1=P$
is the Sylow $p$-subgroup of $I$.  In our case $I=\langle \sigma\rangle$ is
cyclic of order $6$ and  $P\subset I$ is generated by the
element $\sigma^3.$ A simple computation using (\ref{Gammaeq}) shows that
\[
   \Gamma_1=\Gamma_2=\Gamma_3=P, \quad \Gamma_4=\{1\}.
\]
Theorem \ref{ssredthm3} implies that
\begin{equation} \label{deltaeq2}
  \delta = 2(g_{\Yb}-g_{\Zb^w}),
\end{equation}
where $g_{\Yb}$ (resp.\ $g_{\Zb^w}$) is the arithmetic genus of $\Yb$ (resp.\
of the quotient curve $\Zb^w:=\Yb/P$). 

The curve $\Yb$ has genus $3$. The
computations of \S~\ref{exa23} show that the curve $\Zb^w$ is a
semistable curve over $\FF_4$ with three smooth irreducible components
$\Zb_0^w,\Zb_1^w,\Zb_2^w$, where $\Zb_1^w$ and $\Zb_2^w$ each
intersect $\Zb_0^w$ in a unique point.   The curve $\Zb_0^w$ is
canonically isomorphic to the genus-one curve $\Yb_0$ (since $I$ acts
trivially on $\Yb_0$), while $\Zb_1^w$ and $\Zb_2^w$ are curves of
genus zero. We conclude that
$g({\Zb^w})=1$, and hence $\delta=4$ by \eqref{deltaeq2}. All in all we obtain
\[
     f_{Y/K} = \epsilon+\delta=4+4=8.
\]

\bigskip\noindent {\em Acknowledgment} We would like to thank Tim Dokchitser
for suggesting the problem motivating this paper and for many helpful
conversations and useful comments. We also want to thank him for inviting
us to Bristol, where some parts of this paper were written. Furthermore, we
would like to thank Qing Liu for a helpful conversation on the proof of Theorem
\ref{ssredthm2}, and the referee for the detailed report.

\vspace{5ex}\noindent {\small Irene Bouw, Stefan Wewers\\ Institut
  f\"ur Reine Mathematik\\ Universit\"at
  Ulm\\ Helmholtzstr.\ 18\\ 89081 Ulm\\ {\tt irene.bouw@uni-ulm.de,
    stefan.wewers@uni-ulm.de}}

\end{document}